\documentclass[11pt]{amsart}
\usepackage{times}
\usepackage[T1]{fontenc}
\usepackage{amssymb, amsthm, amsmath}
\usepackage{mathrsfs}
\usepackage{centernot}
\usepackage{mathtools}
\usepackage{stmaryrd}

\usepackage[active]{srcltx}

\usepackage{todonotes}

\usepackage{setspace}
\usepackage{geometry}

\usepackage{hyperref}

\DeclareMathOperator{\acl}{acl}
\DeclareMathOperator{\dcl}{dcl}

\DeclareMathOperator{\eq}{eq}

\DeclareMathOperator{\tp}{tp}

\DeclareMathOperator{\mr}{RM}

\newtheorem{introtheorem}{Theorem}

\newtheorem{theorem}{Theorem}[section]

\newtheorem{claim}{Claim}[theorem]

\newtheorem{corollary}[theorem]{Corollary}

\newtheorem{fact}[theorem]{Fact}
\newtheorem{lemma}[theorem]{Lemma}

\newtheorem{proposition}[theorem]{Proposition}

\newtheorem*{gen-dif}{\fbox{{\large A}} \hypertarget{Agen-dif}{Gen-Dif}}

\newtheorem*{min-balln}{\fbox{{\large A}} \hypertarget{Amin-ball}{Cballs}}

\theoremstyle{definition}
\newtheorem{definition}[theorem]{Definition}
\newtheorem{example}[theorem]{Example}
\newtheorem{remark}[theorem]{Remark}
\newtheorem{question}[theorem]{Question}
\newtheorem{notation}[theorem]{Notation}

\newcommand{\Cc}{{\mathbb{C}}}
\newcommand{\Rr}{{\mathbb{R}}}
\newcommand{\Nn}{{\mathbb{N}}}
\newcommand{\Qq}{{\mathbb{Q}}}

\newcommand{\Zz}{{\mathbb {Z}}}

\newcommand{\m}{\textbf{m}}
\newcommand{\bk}{\textbf{k}}

\newcommand{\CD}{{\mathcal D}}
\newcommand{\CL}{{\mathcal L}}
\newcommand{\CK}{{\mathcal K}}

\newcommand{\CM}{{\mathcal M}}

\newcommand{\CC}{{\mathcal C}}

\newcommand{\CO}{{\mathcal O}}

\newcommand{\rest}{\upharpoonright}

\renewcommand{\phi}{\varphi}

\def\bm{\m}

\def\dpr{\mathrm{dp\text{-}rk}}
\def\sub{\subseteq}
\def\sups{\supseteq}
\def\drk{\mathrm{D\text{-}rk}}
\def\adrk{\mathrm{aD\text{-}rk}}
\def\KOp{(K/\CO)_{\text{\tiny{p-adic}}}}

\newenvironment{claimproof}[1][\proofname]
{%
	\proof[#1]%
	\renewcommand	*\qedsymbol{$\square$ (claim)}%
}
{%
	\endproof%
}

\date{\today}
\title{The infinitesimal subgroup of interpretable groups in some dp-minimal valued fields}

\author{Yatir Halevi}
\address{Department of Mathematics, Ben Gurion University of the Negev, Be'er-Sheva 84105, Israel}
\curraddr{Faculty of Mathematics\\Technion - Israel Institute of Technology\\Haifa\\ Israel.}
 \email{yatirh@math.technion.ac.il}

\author{Assaf Hasson}
\address{Department of Mathematics, Ben Gurion University of the Negev, Be'er-Sheva 84105, Israel}
\email{hassonas@math.bgu.ac.il}

\author{Ya'acov Peterzil}
\address{Department of Mathematics, University of Haifa, Haifa, Israel}
\email{kobi@math.haifa.ac.il}

\begin{document}
	
	   \thanks{The first author was partially supported by ISF grants No. 555/21 and 290/19. The second author was supported by ISF grant No. 555/21. The third author was supported by ISF grant No. 290/19.}
	\begin{abstract}
		
		We continue our local analysis of groups interpretable in various dp-minimal valued fields, as introduced in \cite{HaHaPeGps}. We associate with every infinite group $G$ interpretable in those fields an infinite type-definable infinitesimal subgroup $\nu(G)$, generated by the four infinitesimal subgroups $\nu_D(G)$ associated with the distinguished sorts $K$, $\bk$, $\Gamma$ and $K/\CO$. 
		To show that $\nu(G)$ is type-definable, we show that the resulting subgroups $\nu_D(G)$ commute with each other as $D$ ranges over the four distinguished sorts.  We then study the  basic properties of $\nu(G)$. Among others, we show that $\nu(G_1\times G_2)=\nu(G_1)\times \nu(G_2)$ and that if $G_1\le G$ is a definable subgroup then $\nu(G_1)$ is relatively definable in $\nu(G)$. We also discuss possible connections between $\dpr(\nu(G))$ and  elimination of imaginaries. 
		
	\end{abstract}
	
	\maketitle	
	
	\tableofcontents{\setcounter{tocdepth}{0}}

	\section{Introduction}
	We revisit the construction of infinitesimal subgroups of groups interpretable  in $V$-minimal expansions of algebraically closed valued fields, $T$-convex expansions of power bounded o-minimal fields and in (analytic expansions of) $p$-adically closed fields.\footnote{See Remark \ref{R: remark on Qp} for expansions of $p$-adically closed fields.} We show that this construction, introduced in \cite{HaHaPeVF} and studied axiomatically in \cite{HaHaPeGps}, gives cleaner results when restricted to the more concrete setting described above. 
    
	To recall, in \cite{HaHaPeVF, HaHaPeGps, HaHaPeSemisimple}, we introduced the construction, for a certain class of interpretable groups, of four \textit{infinitesimal subgroups}.
	Each of these infinitesimal subgroups of an interpretable group $G$ is associated with one of four  \textit{distinguished sorts} and can be viewed as a trace of the sort inside $G$: The valued field $K$, the value group $\Gamma$, the residue field $\bk$ (when infinite) or  $K/\CO$, the cosets of the valuation ring.  At least one of the four infinitesimal groups has positive dp-rank. Below, these are denoted $\nu_D$, as $D$ ranges over these distinguished sorts. 
    
    Although these type-definable subgroups of $G$ are ``small'' they often reflect  group theoretic properties of $G$. For example, if $\nu_{\Gamma}$ is nontrivial then $Z(G)$ is infinite and if $\nu_{K/O}$ is nontrivial then $G$ contains a definable infinite normal abelian subgroup (see \cite[Proposition 6.1, Proposition 7.1]{HaHaPeSemisimple}).  Further applications include the proof in \cite{HaHaPeSemisimple} that a non-abelian definably simple group $G$ is definably isomorphic to either a $K$-linear group or to a $\bk$-linear group. In \cite{HaHaPeVF} and in \cite{HaHaPeGps} infinitesimal subgroups were crucial in characterizing interpretable fields and for studying dp-minimal interpretable groups in those classes of valued fields.

	The key observation underlying these results is the fact that every infinite group $G$ has an infinite definable subset $X\subseteq G$ which is \emph{almost strongly internal} to at least one of the distinguished sorts. Namely,
	there is a distinguished sort $D$ and a definable function $f: X\to D^n$ (some $n$) with finite fibers. For a fixed distinguished sort $D$, the \emph{almost $D$-rank of $G$}, denoted by $\adrk(G)$, is the maximal dp-rank of such a set $X$. In  fact, in Proposition \ref{P:lasi=lsi} we clarify the situation by proving that in our setting, for all distinguished sorts $D$, except  $D=K/\CO$ in the $p$-adic case, we can realize $\adrk(G)$ with a definable \textbf{injection} $f:X\to D^n$.

	The main result of the present paper is: 
	
	\begin{introtheorem}\label{T: intro-main}
	Let $\CK$ be an expansion of a valued field that is either (i) $V$-minimal, (ii) power bounded $T$-convex or (iii) pure $p$-adically closed, and let $G$ be an infinite interpretable group in $\CK$. 
	\begin{enumerate}
		\item To each distinguished sort $D$ the collection of type-definable subgroups almost strongly internal to $D$ and of dp-rank $\adrk(G)$ has a minimal element, denoted by $\nu_D:=\nu_D(G)$.
		\item 


The subgroup generated by the four $\nu_D(\widehat \CK)$\footnote{Formally, one should have written $\nu_D(G)(\widehat \CK)$.} in a sufficiently saturated model $\widehat \CK\succ \CK$ is infinite, type definable, and equals the \textbf{internal} direct product \[\nu(G)(\widehat \CK):=\nu_K(\widehat \CK)\times \nu_\Gamma (\widehat \CK) \times \nu_{K/\CO}(\widehat \CK)\times \nu_\bk(\widehat \CK).\]
        In particular, the four type-definable subgroups commute with each other.



	\end{enumerate}
\end{introtheorem}
The proof of the above appear in Section \ref{ss:infin for all} and Theorem \ref{T:main}.

    This theorem improves the results of \cite{HaHaPeGps} in two main aspects. First, it replaces the construction of the four infinitesimal subgroups with a single type-definable group $\nu(G)$ capturing the exact same information. Second, it solves one of the technical drawbacks of our previous work: To construct the groups $\nu_D$ we had, in general, to first pass to a quotient by a finite normal subgroup. Here, we use Proposition \ref{P:lasi=lsi}, as noted above, to eliminate the need for this quotient in all cases but $D=K/\CO$ in the $p$-adic case. In the remaining case,  $D=K/\CO$ in the $p$-adic case, we show that  the normal finite subgroup is canonical (Proposition \ref{Canonical H} and  see also Fact \ref{F: H contained if same adrk}).

	The rest of the paper is dedicated to studying the group $\nu(G)$. In \cite{HaHaPeGps} we asked (in slightly different terms), whether $\dpr(\nu(G))=\dpr(G)$. In Example \ref{E: exp example} we show that this need not be true in the context of $V$-minimal analytic expansions of ACVF$_{0,0}$. We then apply this example to give a short geometric proof of \cite[Theorem 1.1]{HaHrMac3}, asserting that analytic expansions of ACVF$_{0,0}$ do not eliminate imaginaries down to the so-called geometric sorts of \cite{HaHrMac1}. It may be worth noting that, assuming elimination of imaginaries down to 0-dimensional sorts (e.g., in pure ACVF$_{0,0}$, RCVF or $p$-adically closed fields), partial positive solutions to the above question exist. This possible connection with questions of elimination of imaginaries (see also \cite[Question 1.5]{HaHrMac3}) remains unclear, and seems worthy of further investigation. 
	
	In the last two sections of the paper we initiate a preliminary investigation of the correspondence $G\mapsto \nu(G)$. We show, among others: 
	\begin{introtheorem}
		Let $\CK$ be an expansion of a valued field of characteristic $0$ that is either (i) $V$-minimal, (ii) power bounded $T$-convex or (iii) pure $p$-adically closed and let $G_1,G_2$ be interpretable infinite groups. Then: 
		\begin{enumerate}
			
			\item  If $G_2\le G_1$ is a subgroup then $\nu(G_2)$ is relatively definable subgroup in $\nu(G_1)$. Moreover, if $\dpr(G_1)=\dpr(G_2)$ then $\nu(G_1)=\nu(G_2)$ (Corollary \ref{C: nu rel def}, Proposition \ref{P: nu is local}).
			\item  $\nu(G_1\times G_2)=\nu(G_1)\times \nu(G_2)$ (Proposition \ref{P:nu in products}).
			
			\item   If $f: G_1\to G_2$ is a definable surjective homomorphism with finite kernel then $f$ maps $\nu(G_1)$ onto $\nu(G_2)$ (Proposition \ref{P:nu under finite to one}). 
		\end{enumerate}
	\end{introtheorem}
	
	The proof of the results above relies on showing the corresponding results for the different distinguished sorts and, in fact, for those we sometimes have  stronger results. 
	 \\
	
	\noindent{\bf Acknowledgment} 
	We would like to thank Raf Cluckers, Deirdre Haskell  and Moshe Kamensky for helpful discussions regarding Example \ref{E: exp example}.  We would especially like to thank the anonymous referee for a truly exceptional report that helped improve the exposition of the paper considerably.  \\

    \noindent{\bf AI statement} No AI was used during the preparation of this paper.

	\section{Background and preliminaries}
	In the present section we review the main objects of study introduced in \cite{HaHaPeVF}, \cite{HaHaPeGps} and \cite{HaHaPeSemisimple}, and their basic properties. The original work was carried out in a general axiomatic setting. Since we work here with concrete families of valued fields some  definitions can be simplified. In those cases, we provide the simpler definition, and prove its equivalence to the original one. 
	
	In the first two subsections, we review standard notation and terminology. We then proceed to discuss more specialized terminology. 
	
	\subsection{Notation and other conventions}
	Our notation and terminology is standard and similar to that of \cite{HaHaPeVF}, \cite{HaHaPeGps}, so we will be brief. We refer the reader to those papers for a more detailed technical introduction. 
	
	Throughout $\CK$  denotes an expansion of a  valued  field of characteristic $0$ in a language $\CL$ expanding the language of valued rings. We denote its universe by $K$. We assume $\CK$ to be  $(|\CL|+2^{\aleph_0})^+$-saturated. We also fix a $|K|^+$-saturated elementary extension $\widehat \CK\succ \CK$. Definable (and type definable, see below) sets will always be assumed to be defined in $\CK^{eq}$ over sets of parameters from $\CK^{eq}$. 
	
	\textbf{Thus, from now on we do not distinguish interpretable sets from definable ones, and refer to both classes of sets as \textit{definable}.} In $\CK^{eq}$ definable sets have codes and for a definable set $X$ we let $[X]$ denote the $\dcl^{\eq}$-closure of some (equivalently, any) code for $X$. 
	
	By a (partial) type we mean a consistent collection of formulas. Two partial types $\rho_1, \rho_2$ are equal, denoted $\rho_1\equiv\rho_2$, if they are logically equivalent. We write $\rho \vdash X$ to mean that $\rho$ is concentrated on $X$ and write $X\in \rho$ to mean that $X$ is a formula in $\rho$ (in particular $\rho\vdash X$).  If $f:X\to Y$ is a definable function and $\rho\vdash X$ is a partial type then we let $f_*\rho\vdash Y$ denote its image. A {\em type-definable subgroup} of a definable group $G$ is a type  $\nu\vdash G$ over $\CK$  such that $\nu(\widehat \CK)$ is a subgroup of $G(\widehat \CK)$.  When we refer to group theoretic properties of $\nu$ such as ``$\nu$ is torsion-free'' or ``$\nu$ is commutative'' we mean that $\nu(\widehat \CK)$ satisfies these properties.

	We assume familiarity with dp-rank and its properties, such as sub-additivity, invariance under definable finite-to-finite correspondences, invariance under automorphisms etc. See the preliminaries sections of \cite{HaHaPeVF,HaHaPeGps} for a more detailed discussion.

	Closely related to the dp-rank in dp-minimal structures is the {\em $\acl$-dimension}: $\dim_{\acl}(a/A)$ is the minimal  $k\leq |a|$ for which there exists a sub-tuple
	$a'\sub a$ of size $k$, such that $a\in \acl(a'A)$. All  distinguished sorts in our settings have \textit{algebraic dp-rank}, meaning that dp-rank and $\dim_{\acl}$, as defined above, agree (\cite[Fact 4.7]{HaHaPeGps} for unstable $D$, and when $D=\bk$ is an ACF it is well-known). In dp-minimal structures with algebraic dp-rank, $\dim_{\acl}(a)$ can also be characterized as the maximal size of a sub-tuple $a'\sub a$ that is $\acl$-independent (i.e., no element of $a'$ is algebraic over the rest of $a'$), see \cite[Lemma 2.8]{HaHaPeSemisimple}. Thus, the above definition coincides with the definition of $\dim_{\acl}$ in \cite{HaHaPeGps}. 
	
	The {\em $\acl$-dimension}  of an $A$-definable set $X$, $\dim_{\acl}(X)$, is defined as the maximum of $\dim_{\acl}(a/A)$, for all $a\in X$.\footnote{This does not depend on $A$ when dp-rank is algebraic.} 
	In {\em geometric structures}, $\dim_{\acl}$ coincides with the standard dimension associated with the $\acl$-pregeometry. However, in the current paper $\acl$ in the distinguished sorts (e.g., $K/\CO$ in all settings) need not satisfy the Exchange Property, and some of the distinguished sorts ($\Gamma$ and $K/\CO$ in the $p$-adically closed case) do not satisfy uniform finiteness.

	\subsection{Valued fields}
	For any valued field $(K,v)$, we let $\Gamma$ denote the value group, $\CO$ the valuation ring,  $\m$  the maximal ideal and $\bk:=\CO/\m$ the residue field. The  \emph{distinguished sorts} of $\CK$ are  $K$, $\Gamma$, $\bk$ and $K/\CO$. 
	
	For the rest of the paper,  $\CK$ is assumed to be one of the following (see \cite[Section 2.3]{HaHaPeVF} for details): 
	\begin{enumerate}
		\item A $V$-minimal valued field (such as ACVF$_{0,0}$ and its analytic expansions).
		\item A power-bounded $T$-convex valued field (e.g., RCVF and $\mathbb{R}_{an}$-conv).
		\item A $p$-adically closed field. 
	\end{enumerate}

	\begin{remark}\label{R: remark on Qp}
		As in our previous works, \cite{HaHaPeVF, HaHaPeGps, HaHaPeSemisimple}, in the $p$-adically closed setting, the proofs here go through as written in $P$-minimal expansions (\cite{HasMac}) that are 1-h-minimal (in the sense of \cite{hensel-minII}) and have definable Skolem functions in the valued field sort. 
		Namely, the assumption that the field sort is a pure valued field is never used.
		In particular, the results of the current paper remain true, as stated and as proved, in models of $\Qq_{p, an}$, the expansion of $\Qq_p$ by all analytic functions on $\mathbb Z_p^n$ (for all $n$).

	\end{remark}

	\begin{notation}
		As in our previous work, our goal is to treat the above three settings as uniformly as possible. However, as will be evident below, in the $p$-adically closed setting the sorts $\Gamma$ and especially $K/\CO$  often require  different statements and proofs. %
		In the latter case we introduce the notation $D=\KOp$. 
	\end{notation}
	
	Note that, in all our settings, $\CK$ is dp-minimal. Since we are working solely in $\CK^{\eq}$, all definable sets have finite dp-rank.

	A {\em ball} in $K$ is a set of the form $$B_{\geq \gamma}(a)=\{x\in K: v(x-a)\geq \gamma\}\,\, \mbox{ or } \,\, B_{>\gamma}(a)=\{x\in K: v(x-a)> \gamma\},$$ for $\gamma\in \Gamma$. Since the valuation descends naturally to $K/\CO\setminus \{0\}$ analogous notions exist in $K/\CO$. However, note that in $\KOp$ the set $B_{\geq \gamma}(a)$ is finite for $\gamma\in \mathbb Z$,  \textbf{so here we will reserve the term ``ball'' to infinite balls only}. A ball in $K^n$ or in $(K/\CO)^n$ is a product of balls of equal radius.
	
	The value group is either a $\mathbb{Z}$-group (if $\CK$ is $p$-adically closed) or a densely ordered abelian group (in all other cases). In order to allow uniform treatment of definable sets in $\Gamma$ we define the following:
	
	\begin{definition}\label{D:box}
		A subset $B\subseteq \Gamma^n$ is called {\em a $\Gamma$-box around $a=(a_1,\dots,a_n)$} if it is of the following form:
		\begin{enumerate}
			\item When $\Gamma$ is dense:  $\prod_{i=1}^n (b_i,c_i)$ for some $b_i<a_i<c_i$ in $\Gamma$.
			\item  When $\Gamma$ is a $\mathbb{Z}$-group:    A cartesian product of $n$ sets of the form $(b_i,c_i)\cap \{x_i: x_i\equiv a_i \mod m_i\}$, $m_i\in \mathbb N$, $b_i<a_i<c_i$, \textbf{where both intervals $(b_i,a_i)$ and $(a_i,c_i)$ are infinite}.
		\end{enumerate}
	\end{definition}
	Note that the $\KOp$-balls and the $\Gamma$-boxes around $a$ are not uniformly definable in families because of the lack of uniform finiteness.
	
	When $\CK$ is $p$-adically closed, it is elementarily equivalent (as a pure valued field) to some finite  extension  $\mathbb{F}$ of $\mathbb{Q}_p$.  By saturation, we may assume that $(K,v)$ is an elementary extension of $(\mathbb{F},v)$.  Since the value group $\Gamma_{\mathbb{F}}$ is isomorphic to $\mathbb{Z}$ we identify $\Gamma_{\mathbb{F}}$ with $\mathbb{Z}$ and view it as a prime (and minimal) model for $\Gamma$. In order to be consistent with the conventions of \cite{HaHaPeGps} and apply its results freely, when $\CK$ is $p$-adically closed we name each element of $\mathbb{F}$ in the language.

	\subsection{Strong internality and related concepts}
	The starting point of all of our work is the following observation \cite[Lemma 7.3, Lemma 7.6, Lemma 7.10]{HaHaPeGps}: 
	\begin{fact}\label{F: alsi to D}
		Let $X$ be an infinite set definable in $\CK$. Then there exists an infinite subset $Y\subseteq X$ and a finite-to-one function $f: Y\to D^m$ (both definable, possibly, over additional parameters), for some $m\in \mathbb N$,  where $D$ is one of the sorts: $K$, $\Gamma$, $\bk$ or $K/\CO$.
	\end{fact}
	
	As in \cite{HaHaPeGps,HaHaPeVF,HaHaPeSemisimple}, the \emph{distinguished sorts} of $\CK$ are the sorts $K$, $\Gamma$, $\bk$ and $K/\CO$. We tacitly disregard $\bk$ when $\CK$ is $p$-adically closed. We recall here some of their main properties.
	
	All the distinguished sorts, except $K/\CO$, are stably embedded. Except for $\Gamma$ and $K/\CO$ in the p-adic case and $\bk$ in the $V$-minimal case, they are all SW-uniformities (in the sense of \cite{SimWal}), i.e., dp-minimal uniform structures that are $t$-minimal, see also \cite{HaHaPeVF} for their main properties.\footnote{Recall that a structure $\CM$ with a definable basis for a topology on $M$ is $t$-minimal if for any definable $S\sub M$ the interior of $S$ is non-empty if and only if $S$ is infinite.}
	The sorts $\Gamma$ and $\bk$ eliminate imaginaries in all cases. The residue field $\bk$ in the $V$-minimal setting is a pure algebraically closed field and  the only stable distinguished sort.

	Fact \ref{F: alsi to D}  prompted the following definitions: 
	\begin{definition}\label{def: internalities}
		Let $S$ and $D$ be  definable sets in $\CK$. 
		\begin{enumerate}
			\item 	$S$ is \emph{almost strongly internal to $D$} if there exists a definable finite-to-one function $f: S\to D^n$, for some $n$. 
			\item $S$ is \emph{strongly internal to $D$ } if there exists a definable injective $f: S\to D^n$, for some $n$. 
			\item A (partial) type $p$ is almost strongly internal to $D$ if there is a definable set $S$ such that $p\vdash S$ and $S$ is almost strongly internal to $D$. Similarly, for ``strongly internal''.
			\item $S$ is \emph{locally (almost) strongly internal to $D$} if there exists {\bf an infinite} $X\sub S$ that is (almost) strongly internal to $D$. 
			\item The \emph{$D$-critical rank of $S$} is 
			\[
			\drk(S):=\max\{\dpr(X): X\sub S, \,\, X \text{ is strongly internal to } D\},
			\]
			if $S$ is locally strongly internal to $D$, and $0$ otherwise.
			
			The \emph{almost $D$-rank of $S$}, $\adrk(S)$ is defined analogously. 
			
			\item For $S$   locally strongly internal to $D$, a definable set $X\sub S$ is \emph{$D$-critical with respect to $S$} if $X$ is strongly internal to $D$ and $\dpr(X)=\drk(S)$ .
			
			\item For $S$   locally almost strongly internal to $D$, let $m_{\text{crit}}\in \mathbb N$ be the minimal $m$ such that there exists an $m$-to-one definable $f:X\to D^n$, for some $n$, as $X$ varies over all subsets of $S$ with $\dpr(X)=\adrk(S)$. A definable set $X\sub S$ is \emph{almost $D$-critical with respect to $S$} if 
			\begin{itemize}
				\item $\dpr(X)=\adrk(S)$, and
				\item there exists a definable $m_{\text{crit}}$-to-1 function $f: X\to D^n$, for some $n$.

			\end{itemize}
		\end{enumerate}
	\end{definition}
	Note that a $D$-critical set need not be almost $D$-critical.

	In \cite{HaHaPeGps} the construction of the infinitesimal subgroups takes place in the abstract context of \emph{vicinic sets}. In the more specialized setting of our valued fields, some of the definitions can be strengthened and some of the arguments can be simplified. 
	
	As a first step, we slightly strengthen \cite[Remark 2.12(2)]{HaHaPeSemisimple}. In that remark it was observed that 
	given an (almost) $D$-critical subset $X$ of a definable group $G$, we may, at the cost of shrinking $X$ but not its dp-rank, require that the function witnessing it $f: X\to D^m$ is such that $m=\dpr(X)$. 
	
	To state the result, we recall the following: 
	\begin{definition}
		Let $X$ be an $A$-definable set in $\CK$. A point $a\in X$ is \emph{generic over $B\sups A$} (or $B$-generic), if $\dpr(a/B)=\dpr(X)$. By referring to ``a generic $a\in X$'' we mean that $a$ is $[X]$-generic. 
	\end{definition}
	
	\textbf{When referring to a definable set $X$ and saying that it is (almost) strongly internal to $D$ over $A$ we are tacitly assuming that not only is the set $X$  is $A$-definable, but so is some function $f: X\to D ^m$ witnessing (almost) strong internality of $X$. Moreover, if for such a set $X$ we say that $a\in X$ is generic, without specifying the parameters set $A$, we are assuming that $a$ is generic over a $[X][f]$, for some $f$ witnessing this.}	\\
    
	Next we note the following: 
	\begin{lemma}\label{L:onto open}
		Let $D$ be an unstable distinguished sort in $\CK$, $S\sub D^n$ an $A$-definable set of dp-rank $m$ and $s\in S$ an $A$-generic.
		Then there exist $B\supseteq A$, a $B$-definable $S'\sub S$ and a coordinate projection $\pi:D^n\to D^m$ such that $s\in S'$, $\dpr(s/B)=m$ and the restriction of $\pi$ to $S'$ is injective. 
	\end{lemma}
	\begin{proof}
		When $D$ is an SW-uniformity, this is \cite[Lemma 4.6]{SimWal} combined with \cite[Corollary 3.12]{HaHaPeVF}. So we only have to verify the lemma if $\CK$ is $p$-adically closed and $D=\Gamma$ or $D=K/\CO$. In the former case, this follows from cell-decomposition for Presburger Arithmetic, \cite{ClucPresburger}.
		
		Assume that $D=\KOp$. By \cite[Lemma 3.30]{HaHaPeGps}, there is $A'\supseteq A$ and an $A'$-definable coset $s+H\sub S$, $H\leq (K/\CO)^n$,  with $\dpr(s/A')=\dpr(S)$ (the ``minimal fibers'' assumption in that statement is verified by the aid of \cite[Lemma 3.26]{HaHaPeGps}).
		By  \cite[Lemma 3.10, Remark 3.11]{HaHaPeSemisimple}, there is a definable subgroup $H_1\sub H$, $\dpr(H_1)=\dpr(H)$, projecting injectively onto an open ball in $(K/\CO)^m$. Thus, $s+H_1$ projects injectively onto an open $W \sub (K/\CO)^m$. 
		Finally, by \cite[Proposition 3.8(1)]{HaHaPeGps}, we may replace $W$ by a ball $V\sub W$, definable over some $B\supseteq A'$ so that $\dpr(s/B)=m$. The pre-image of $V$ inside $S$ is the desired $S'$. 
	\end{proof}
	
	\begin{remark}\label{R: onto open for K/O padic}
		We note for later use that in the case that $D=\KOp$  (the proof of) Lemma \ref{L:onto open} actually gives that $S'$ is a coset of a definable subgroup  whose image  under $\pi$ is a ball in $D^m$.
	\end{remark}
	
	Consequently, we get the following:  
	\begin{corollary}\label{C: has D-sets}
		Let $G$ be a definable group in $\CK$, $D$ an unstable distinguished sort, and $X\sub G$ an (almost) $D$-critical set of dp-rank $m$, all definable over $A$. Then for any $A$-generic $a\in X$ there exists a $B$-definable (almost) $D$-critical $Y\sub X$, for $B\supseteq A$, $a\in Y$, witnessed by a $B$-definable function $g: Y\to D^m$, such that $a$ is $B$-generic in $Y$.
	\end{corollary}

	\section{$D$-balanced groups}		
	In the context of groups definable in $\CK$ some technical notions and concepts introduced in \cite{HaHaPeGps} can be circumvented or simplified and thus giving a more direct construction of the infinitesimal subgroups. Let $D$ be a distinguished sort.

	We first introduce our main notion of interest for this section. A key result of \cite{HaHaPeGps} was that given a definable group $G$ in $\CK$ there exists a finite normal subgroup $H\trianglelefteq G$, such that $\drk(G/H)=\adrk(G/H)$, and in addition these are equal to $\adrk(G)$. It turns out that this dp-rank equality is useful for much that follows, so we define: 
	
	\begin{definition}
		A definable group $G$ in $\CK$ is \emph{$D$-balanced} if $\drk(G)=\adrk(G)$. 
	\end{definition}
	
	Notice that if $\adrk(G)=0$ then $G$ is $D$-balanced.
	Also, by definition, in a $D$-balanced group which is locally almost strongly internal to $D$, every almost $D$-critical set is $D$-critical. Indeed, since $G$ is $D$-balanced there is $Y\subseteq G$ which is strongly internal to $D$ of dp-rank $\adrk(G)$ and so $m_{\text{crit}}=1$; thus every almost $D$-critical set is $D$-critical. 
    
    We are now ready to prove an important result in this paper, namely that in all but one case, the finite group $H$ above can be chosen to be trivial. We first recall the following observation which will be used frequently.

    \begin{fact}\label{F:finite index same adrk}\cite[Lemma 2.18(1)]{HaHaPeSemisimple}
        Let $G$ be an infinite definable group in $\CK$ and $D$ a distinguished sort. For any finite normal subgroup $H\trianglelefteq G$, $\adrk(G/H)=\adrk(G)$.
    \end{fact}

	\begin{proposition}\label{P:lasi=lsi}
		Let $G$ be an infinite definable group in $\CK$,  $D$ a distinguished sort other than $\KOp$. Then $G$ is $D$-balanced.
	\end{proposition}
	\begin{proof}
		We may assume that $\adrk(G)>0$. By  \cite[Corollary 4.2]{HaHaPeSemisimple}, if $D=K$ (in all settings) then $G$ is $D$-balanced. If $\CK$ is power-bounded $T$-convex, so weakly o-minimal, then $\dcl^{\eq}=\acl^{\eq}$ (\cite[Lemma 2.16]{MelRCVFEOI}). It follows that $G$ is $D$-balanced, see \cite[Remark 7.5]{HaHaPeGps} for all $D$ in this setting.
		
		Thus, we are left with the following cases:  $\CK$ is $V$-minimal and $D\neq K$, or $D=\Gamma$ for $\CK$ a $p$-adically closed field. For what follows, note that if $G_1$ is a finite index subgroup of $G$ then $\drk(G)=\drk(G_1)$ and $\adrk(G)=\adrk(G_1)$; so there is no harm in passing to a finite index subgroup.  
		
		\begin{claim}\label{C:dim of G^n}
			After possibly passing to a finite index subgroup of $G$, there exist a finite central subgroup $H\trianglelefteq G$ and a type-definable subgroup $\rho$ of $G/H$, strongly internal to $D$, such that 
			\begin{enumerate}
				\item $G/H$ is $D$-balanced and $\drk(G/H)=\dpr(\rho)$
				\item For any natural number $n$, $\dpr(\rho)=\dpr(\rho^{(n)})$ where $\rho^{(n)}(\widehat \CK)=\{a^n :a\in \rho(\widehat \CK)\}.$
			\end{enumerate}
		\end{claim}
		\begin{claimproof}
			We break the proof into two cases.
			
			\underline{Case 1:} Assume that $D=\bk$ and that $\CK$ is $V$-minimal. By \cite[Proposition 6.2]{HaHaPeGps}, there exists a  finite normal subgroup $H\trianglelefteq G$ and an $\omega$-stable connected definable subgroup $\widehat \rho$ of $G$ containing $H$, satisfying that for $\rho=\widehat \rho/H$, $\dpr(\widehat \rho)=\adrk(G)$ and $\rho$ is strongly internal to $D$. It follows, using Fact \ref{F:finite index same adrk} for the left most inequality, that
            \[\adrk(G)=\adrk(G/H)\geq \drk(G/H)\geq \dpr(\rho)=\dpr(\widehat \rho)=\adrk(G),\] which gives that $G/H$ is $D$-balanced.
            Moreover, as given in \cite[Proposition 6.2]{HaHaPeGps}, $\rho$ is a definable subgroup of $G/H$ which is definably isomorphic to a connected algebraic group in $\bk$, an algebraically closed field of characteristic $0$. 
            
            The algebraic group  $\widehat \rho$ is connected and thus must be contained in the finite index subgroup $C_G(H)\leq G$ hence $H$ is contained in $C_G(H)$ as well. We may thus replace $G$ by $C_G(H)$ and assume that $H$ is central in $G$. As we remarked before, it follows that $C_G(H)/H$ is $D$-balanced as well.
            
			Since $\rho$ is an algebraic group in $\mathrm{ACF}_0$, it is  well known that $\rho^{(n)}$ is Zariski dense in $\rho$ hence $\dpr(\rho^{(n)})=\dpr(\rho)$.\footnote{Since the claim is first order, it is enough to prove it in $\Cc$. It is then standard to verify that the differential of $x\mapsto x^n$ at $e$ is $nI$. By the Inverse Function Theorem $x\mapsto x^n$ is locally invertible in the Euclidean topology. Hence, the image has full dimension in $G$, and is therefore Zariski dense in $G$.}
			
			We are reduced to: \\
			
			\underline{Case 2:} $D=K/\CO$ in the $V$-minimal case or $D=\Gamma$ for $\CK$ $p$-adically closed or $V$-minimal.

			By \cite[Theorems 7.8(3,4,5a), 7.12(2,3,4a)]{HaHaPeGps}, there exists a  finite normal subgroup $H\trianglelefteq G$ and a type-definable subgroup $\rho$ of $G/H$ satisfying that $\dpr(\rho)=\adrk(G)=\adrk(G/H)$. A similar argument as in Case 1 gives that $G/H$ is $D$-balanced.
			
            As $H$ is finite, $C_G(H)\leq G$ is a finite index subgroup and by \cite[Proposition 4.35(3)]{HaHaPeGps},\footnote{The reference assumes that $G$ is an \emph{almost $D$-group}. As we shall see in Lemma \ref{L:D-balanced satisfied clause 2}(1), every group locally almost strongly internal to $D$ is an almost $D$-group. The proof of this fact is independent of the current proof.} $H\leq C_G(H)$, i.e. $H$ is central in $C_G(H)$. So we may replace $G$ by $C_G(H)$; note that $C_G(H)/H$ is $D$-balanced as well.
            
            

			By \cite[Theorems 7.8(3,4) 7.12(2)]{HaHaPeGps}, the type-definable subgroup $\rho$ of $G/H$ is definably isomorphic to a type-definable subgroup of $(D^m,+)$ for some natural number $m$. In particular, $\rho$ is commutative. Consider the group homomorphism $x\mapsto x^n$ on $\rho$. As $D$ is torsion-free in these cases, the map is an isomorphism onto its image, so $\dpr(\rho^{(n)})=\dpr(\rho)$.
		\end{claimproof}
		
		We remark for a later use in the proof, that if $D=K/\CO$ for  $V$-minimal  $\CK$, then the proof of the claim shows that $x\mapsto x^n$ is injective on $\rho$.

		We fix $H$ as in the claim, let $|H|=n$ and let $f:G\to G/H$ be the definable surjective quotient homomorphism. By Fact \ref{F:finite index same adrk}, $\adrk(G)=\adrk(G/H)=\drk(G/H)$.
		
		Let $\rho\vdash G/H$ be as provided by the claim 	and consider the definable map $\beta:\rho(\hat \CK)\to G(\hat \CK)$ given by $(f^{-1}(x))^n$.  It is well-defined since $H$ is central of order $n$.
		
		Letting $\widetilde  \rho=\beta_*\rho\vdash G$, the image of the type $\rho$ under $\beta$,  we get that $f_*\widetilde \rho=\rho^{(n)}$. By the claim, $\dpr(\rho^{(n)})=\dpr(\rho)=\drk(G/H)$ and as  $f$ is finite-to-one, $\dpr(\widetilde \rho)=\adrk(G/H)=\adrk(G)$.
		
		Assume first that $D$ is not $K/\CO$, so either $D=\Gamma$ (hence ordered) or $D=\bk$ an algebraically closed field. Either way it eliminates imaginaries. By compactness, $\beta$ extends to a surjective definable map of definable sets, $\beta:X\to Y$, where $X\sub G/H$ is strongly internal to $D$, $\rho\vdash X$ and $\widetilde \rho\vdash Y\sub G$. Since dp-rank is algebraic, we can choose $Y$ so that $\dpr(Y)=\dpr(\widetilde \rho)=\adrk(G)$. By the strong internality of  $X$, we may identify it with a definable subset of $D^l$ for some $l$. By elimination of imaginaries in $D$,  we can use $\beta$ to obtain a definable bijection between $Y$ and a subset of $D^s$, for some $s$, so $Y$ is strongly internal to $D$ and therefore, $\drk(G)\geq \dpr(Y)=\adrk(G)$. Since we always have $\drk(G)\leq \adrk(G)$, equality follows. 
		
		Now assume that $D=K/\CO$ (and $\CK$ is $V$-minimal). Since, as we noted,  $x\mapsto x^n$ is injective on $\rho$, $\beta$ is  injective on $\widetilde \rho(\widetilde \CK)$, indeed if $\beta(x)=\beta(y)$ then by applying $f$, we have $x^n=y^n$ so $x=y$. Thus, as above,  we obtain $X\sub G/H$, $\rho\vdash X$,  strongly internal to $D$ and a  bijective image $Y=\beta(X)\sub G$, and conclude, in this case too, that $\adrk(G)=\drk(G)$, as claimed. 
		
		This concludes the proof that $G$ is $D$-balanced in all cases except $D=\KOp$.
	\end{proof}

	The above proposition fails for  $D=\KOp$, as demonstrated by the following example: 
	\begin{example}\label{E:not balanced}
		In \cite[Example 3.33]{HaHaPeGps} we showed that if $F$ is a quadratic residual extension of $\mathbb Q_p$ and $C_p$ is a cyclic subgroup of $K/\CO$ of order $p$ then $G:=(K/\CO)/C_p$ satisfies $\adrk(G)=1>\drk(G)=0$. In particular, $G$ is not $K/\CO$-balanced. 
	\end{example}

	When $D$ is unstable, the infinitesimal type-definable subgroups were originally constructed for a class of definable groups called ``$D$-groups''. We will not focus on these groups in the present paper, but properties of $D$-groups are quoted in several references. Thus, for completeness, we remind the definition and verify that every $D$-balanced group with $\drk(G)>0$ is a $D$-group.
	
	\begin{definition}\label{D: D-group}
		Let $\CK$ be a valued field as above and $D$ an unstable distinguished sort.  
		\begin{enumerate}
			\item An $A$-definable group $G$ is an \emph{almost $D$-group} if it is locally almost strongly internal to $D$ and satisfies the following property: 
			
			\begin{align*}\tag{$\dag_a$}
				&  \text{For every $B$-definable $X_1, X_2\subseteq G$ almost strongly internal to $D$, with $X_2$} \\ & \text{ almost $D$-critical and $B\supseteq A$ and  for every $B$-generic $(g,h)  \in X_1\times X_2$, we have } \\ & \dpr(g/B,g\cdot h)=\dpr(g/B).
			\end{align*}
			
			\item An $A$-definable group $G$ is a \emph{$D$-group} if it is locally strongly internal to $D$ and satisfies the following property: 
			
			\begin{align*}\tag{$\dag$}
				&  \text{For every $B$-definable $X_1, X_2\subseteq G$  strongly internal to $D$, with $X_2$} \\ & \text{  $D$-critical and $B\supseteq A$ and  for every $B$-generic $(g,h)  \in X_1\times X_2$, we have } \\ & \dpr(g/B,g\cdot h)=\dpr(g/B).
			\end{align*}
			
		\end{enumerate}
	\end{definition}
	
	\begin{remark}
		The above two definitions are needed. We show in \cite{HaHaPeGps} that every definable group locally almost strongly internal to $D$ is an almost $D$-group. However, we give an example, when $D=\KOp$, of a definable group which is locally strongly internal to $D$ but not a $D$-group. See \cite[Example 4.30]{HaHaPeGps}.
	\end{remark}

	\begin{lemma}\label{L:D-balanced satisfied clause 2}
		Let $G$ be a definable group in $\CK$ and $D$ a distinguished sort. 
		\begin{enumerate}
			\item If $G$ is locally almost strongly internal to  $D$ then it satisfies Property $(\dag_a)$.
			
			\item If $G$ is a $D$-balanced group and $\drk(G)>0$ then it satisfies both properties $(\dag)$ and $(\dag_a)$.
		\end{enumerate}
	\end{lemma}
	
	\begin{proof}
		(1) When $D$ is unstable, this is just \cite[Fact 4.25]{HaHaPeGps}. So assume that $D$ is stable, i.e.  $D=\bk$ in the $V$-minimal setting. If $X_1, X_2\sub G$ are almost strongly internal to $\bk$ then by elimination of imaginaries in $\bk$, so is $X_1\cdot X_2\sub G$ and all three sets have finite Morley rank. Since $\bk$ is a strongly minimal structure, dp-rank coincides with Morley rank, so Property $(\dag_a)$ now follows from additivity of Morley rank (=dp-rank) in this case.
		
		(2) Let $G$ be  a $D$-balanced  group with $\drk(G)>0$. Thus, it is locally almost strongly internal to $D$  so satisfies $(\dag_a)$ by (1). We show $(\dag)$ follows immediately.
		
		Since $G$ is $D$-balanced, every $D$-critical set is almost $D$-critical (and obviously every set which is strongly internal to $D$ is almost strongly internal to $D$), so $(\dag)$ is a direct consequence of $(\dag_a)$.
	\end{proof}
	
	\begin{remark}
		Actually, if $D\neq \KOp$ then every definable group locally strongly internal to $D$ satisfies $(\dag)$ (see \cite[Fact 4.25]{HaHaPeGps}). We will not require this result here.
	\end{remark}
	
	The class of $D$-groups is strictly larger than that of $D$-balanced groups of positive $\drk$. However, this only manifests itself when $D=\KOp$:
	\begin{example}\label{E:group but not balanced}
		Let $G$ be the group from Example \ref{E:not balanced}. It is easy to verify that $G\times K/\CO$ is a $(K/\CO)$-group, but since  $G$ itself is only an \textit{almost} $(K/\CO)$-group,  the direct product is not $K/\CO$-balanced. 
	\end{example}

	We shall soon show that for definable groups that are not $D$-balanced, there is a \emph{canonical} finite normal subgroup $H_G$ such that $G/H_G$  is $D$-balanced. To see this, we first recall the following technical result, that will be useful throughout this paper:
	
	\begin{lemma} \label{L:locally critical}
		Let $G$ be a definable group in $\CK$, locally almost strongly internal to a distinguished sort $D$. Let $X,Y\subseteq G$ be such that $X$ is almost strongly internal to $D$ and $Y$ almost $D$-critical and assume that everything is definable over some parameter set $A$. Then for every $(g,h)\in X\times Y$ $A$-generic, $\dpr(X\cap gh^{-1}Y)=\dpr(X)$.
		
		If, in addition,  $\dpr(X)=\adrk(G)$ then $hg^{-1}X\cap Y$ is almost $D$-critical (witnessed by the same function witnessing that $Y$ is almost $D$-critical).
	\end{lemma}
	\begin{proof}
		Let $(g,h)\in X\times Y$ be generic over $A$. By Lemma \ref{L:D-balanced satisfied clause 2}, it is easy to see that  $\dpr(g/A,g\cdot h^{-1})=\dpr(X)$. Since $g\in X\cap gh^{-1}Y$ we get that $\dpr(X\cap gh^{-1}Y)=\dpr(X)$.            The last clause follows, since $hg^{-1}X\cap Y$ is a subset of $Y$ and its dp-rank is equal to $\dpr(X)$.
	\end{proof}
	
	The following is a direct consequence.
	
	\begin{corollary}\label{C:corollary of famous usefull lemma}
		Let $G$ be a definable group in $\CK$ locally almost strongly internal to a distinguished sort $D$. For any  $X\subseteq G$  almost strongly internal to $D$ with $\dpr(X)=\adrk(G)$ there exists a definable $Y\sub X$ that is almost $D$-critical.
	\end{corollary}

	For the following, if $f:X\to Y$ is a definable function and $a\in X$, we set $[a]_f:=f^{-1}(f(a))$. Recall that we denote by $m_{\text{crit}}$ the fiber size of any function witnessing almost $D$-criticality, if $\adrk$ is positive. Otherwise $m_{\text{crit}}=0$. 

	\begin{proposition}\label{Canonical H}
		Let $G$ be a definable group in $\CK$,  $D=\KOp$. Then, there exists a normal abelian subgroup $H_G\trianglelefteq G$, of cardinality $m_{\text{crit}}$, such that $G/H_G$ is $D$-balanced, and $H_G\sub H$ for any finite normal $H$ such that $G/H$ is $D$-balanced. 
	\end{proposition} 
	\begin{proof} If $\adrk(G)=0$ then $H_G=\{e\}$, so assume that $\adrk(G)>0$.
		
		In \cite[Proposition 4.35]{HaHaPeGps}, we associated with every almost $D$-critical definable subset $X\sub G$, a normal, finite, abelian subgroup $H_X\sub G$, such that $|H_X|=m_{\text{crit}}$ is the minimal size of a generic fiber of (any) function $f:X\to D^n$ 
		witnessing almost strong internality in the sense of Definition \ref{def: internalities}(7).  
		Moreover, we showed there that $H_G:=H_X$ does not depend on the choice of $X$,  that $\adrk(G/H_G)=\drk(G/H_G)=\adrk(G)$ and that $H_G$ is abelian. 
		
		Thus, it is left to prove the minimality statement. For the sake of the proof we let $H_1=H_G$,  $\pi_1:G\to G/H_1$  the quotient map,     and let $H_2$ be any other finite normal subgroup such that $G/H_2$ is $D$-balanced, with quotient map $\pi_2: G\to G/H_2$. We show that $H_1\sub H_2$.
		
		For $i=1,2$ let $Y_i\sub G/H_i$ be $D$-critical sets, witnessed by $f_i:Y_i\to D^{m_i}$ and let $X_i:=\pi_i^{-1}(Y_i)$. Note that $\dpr(Y_i)=\dpr(X_i)=\adrk(G)=\adrk(G/H_i)=\drk(G/H_i)$, where the third equality is Fact \ref{F:finite index same adrk}. By the choice of $H_G$ above, $X_1$ is almost $D$-critical,  witnessed by $f_1\circ \pi_1$ (namely, $m_{\text{crit}}=|H_G|$, which is the size of the fibers of $f_1\circ \pi_1$). 
        By Lemma \ref{L:locally critical}, there exists $c\in G$ such that $X:=X_1\cap cX_2$ is almost $D$-critical witnessed by $f_1\circ \pi_1\restriction X$, in particular for any $a\in X$, $[a]_{f_1\circ \pi_1}=aH_1$.
		
		Let $g_i:=f_i\circ \pi_i$ and let $h=(g_1,g_2)$. Then for all $a\in X$ we have $[a]_h=[a]_{g_1}\cap [a]_{g_2}\subseteq aH_1\cap aH_2$. 
		By the minimality of $m_{\text{crit}}$ as defined above, $[a]_h=[a]_{g_1}=aH_1$. So $aH_1\sub aH_2$ and hence $H_1\sub H_2$. Thus, $H_G:=H_1$ is the desired finite normal subgroup.
	\end{proof}
	
	Let us clarify that \cite[Proposition 4.35]{HaHaPeGps} claims the existence of a finite normal subgroup $H$ such that $G/H$ is $D$-balanced, as well as several invariance properties satisfied by $H$. While the proof of that proposition provides us precisely with the group $H_G$ of Proposition \ref{Canonical H} above, the statement of \cite[Proposition 4.35]{HaHaPeGps} does not quite nail the uniqueness of $H_G$. \\
	
	For the sake of uniformity, we define: 
	\begin{notation}\label{N: H_G}
		Let $G$ be a definable group. We let $H_G\trianglelefteq G$ be the unique minimal finite normal subgroup for which $G/H_G$ is $D$-balanced.
		
		By Proposition \ref{P:lasi=lsi}, if $D\neq \KOp$  then $H_G=\{e\}$. In general, by Proposition \ref{Canonical H}, whenever $G$ is $D$-balanced, $H_G=\{e\}$.
	\end{notation}

	We note the following, implicit in \cite[Proposition 4.35]{HaHaPeGps}.
	
	\begin{fact}\label{F: H contained if same adrk}
		Let $G$ be a definable group in $\CK$ and $G_1\leq G$ a definable subgroup with $\adrk(G_1)=\adrk(G)>0$.  Then $H_G\subseteq G_1$.
	\end{fact}
	\begin{proof}
		As noted above, the only non-trivial case is when  $D=\KOp$ (in particular, $D$ is unstable).
		
		Let $X_0$ be an almost $D$-critical subset of $G_1$. By Corollary \ref{C:corollary of famous usefull lemma}, there exists a definable subset $X\subseteq X_0$ that is almost $D$-critical in $G$. In particular, $X$ is almost $D$-critical in $G_1$ as well.
        
        By \cite[Proposition 4.35(1)]{HaHaPeGps} there exists a finite subgroup $H_X\le G_1$ and  $X'\sub X$ of full dp-rank in $X$, such that $X'/H_X$ is strongly internal to $D$. Since $X$ is also almost $D$-critical in $G$ then $H_X=H_G$, by  \cite[Proposition 4.35(2)]{HaHaPeGps}. 
	\end{proof}

	\section{Infinitesimal subgroups} \label{Ss: Infi spgps}
	We next proceed to describing infinitesimal subgroups $\nu_D(G)$  for $D$-balanced groups. For this class of groups the  various descriptions we offer below yield the same group we constructed in \cite{HaHaPeGps}. We then extend the construction to arbitrary definable groups. As in the previous section we are working in the same saturated enough valued field $\CK$.  
	
	\subsection{Infinitesimal subgroups of $D$-balanced groups}
	In \cite{HaHaPeGps} we defined, for a distinguished sort $D$ and a  $D$-group $G$, the infinitesimal subgroup $\nu_D(G)$.  In this section we recall the construction for  $D$-balanced groups and provide (Proposition \ref{P: nu withough generics} below) a geometrically cleaner characterization of these subgroups. We then study some basic properties of $\nu_D(G)$. First recall: 	
	
	\begin{definition} 
		Let $X$ be an $A$-definable set in $\CK$, and let $a\in X$ be $A$-generic. For a set of parameters $B\supseteq A$,  a $B$-definable set $U\sub X$ is a \emph{$B$-generic vicinity of $a$ in $X$} if $a\in U$, and $\dpr(a/B)=\dpr(X)$ (in particular, $\dpr(U)=\dpr(X)$).
	\end{definition}

	\begin{definition}\label{D:vicinities}
		Let $D$ be an unstable distinguished sort in $\CK$, $X$ an $A$-definable set admitting an $A$-definable injection into $D^n$ with $n=\dpr(X)$, and let $a\in X$ be $A$-generic.
		
		The \emph{infinitesimal vicinity of $a$ in $X$}, denoted  $\nu^A_X(a)$, is the partial type consisting of all $B$-definable  $U\sub X$, such that $a\in U$ and $\dpr(a/B)=\dpr(X)$, as $B$ varies over all parameter subsets of $\CK$ containing $A$.
	\end{definition}

	Clarifying remarks are in order.
	
	\begin{remark}
		The current definition varies from the one in \cite{HaHaPeGps} in two aspects and is slightly more general than the original definition.

        \begin{enumerate}
            \item In \cite{HaHaPeGps} (and \cite{HaHaPeGpsCor}), the parameter set was allowed to vary over all subsets of $\CK$ and it was then remarked in \cite[Remark 5.5]{HaHaPeGps} that there is no harm in restricting to subsets containing $A$. However, the argument proving this equivalence has an error inconsequential for the construction of the infinitesimal subgroups. It uses \cite[Lemma 4.13]{HaHaPeGps}, where, although the statement does not specify it, $B$ is required to contain $A$. As we shall note later, this does not affect the construction of the infinitesimal subgroups.

            \item In \cite[Definition 5.4]{HaHaPeGps} infinitesimal vicinities were defined only for a subclass of $D$-critical sets called $D$-sets (\cite[Definition 4.16]{HaHaPeGps}). That definition, applicable only for unstable $D$, required the image of $X$ in $D^n$ to be of particular kind. When $n=\dpr(X)$ this requirement trivially holds. It follows immediately from the definition of $D$-sets and   Corollary \ref{C: has D-sets} that if $X$ is $D$-critical and $a$ is generic in $X$, then there is a generic vicinity $Y$ of $a$ in $X$ which is a $D$-set, and for such $Y$ both definitions of $\nu^A_Y(a)$  coincide. As we shall shortly see  in the Fact \ref{F: vicinities}, we also have $\nu^A_X(a)\equiv \nu^A_Y(a)$. A correction was published in \cite{HaHaPeGpsCor}.
        \end{enumerate}
             
	\end{remark}

	\textbf{In view of the above remark, though $D$-sets are one of the main objects of interest in \cite{HaHaPeGps}, in the present paper we avoid any direct reference to them. Throughout, we systematically replace those with $D$-critical sets.}

	\begin{fact}\label{F: vicinities}
		Let $D$, $X$ and $a\in X$ be as in Definition \ref{D:vicinities}.
		\begin{enumerate}
			\item The definable sets in $\nu^A_X(a)$ form a filter-base.
            \item For any generic vicinity  $Y\subseteq X$ of $a$ in $X$, definable over $B\supseteq A$, $\nu^A_X(a)\equiv \nu^B_Y(a)$.
			\item If $X_1, X_2$ are $A$-definable,  as in Definition \ref{D:vicinities}, then for any $A$-definable function $f:X_1\to X_2$ and an $A$-generic $a\in X_1$,  and $A$-generic $f(a)\in X_2$, we have  $f_*(\nu^A_{X_1}(a))\vdash \nu^A_{X_2}(f(a))$.
			
		\end{enumerate}
	\end{fact}
	\begin{proof}
		(1) \cite[Lemma 5.6]{HaHaPeGps}.
		(2) Application of (1).
		(3)  \cite[Lemma 5.8]{HaHaPeGps}.
	\end{proof}

	As in \cite[Proposition 5.9]{HaHaPeGps}:
	\begin{fact}\label{F: nu in SW for general Y}
		Assume that $D$ is an SW-uniformity, $Y\subseteq D^n$ an infinite definable subset and $a\in Y$ a generic point. Then 
		$\nu_Y(a)\equiv \nu_{D^n}(a)\cap Y\equiv \{V\subseteq Y: V\ni a \text{ is relatively definable open in Y}\}.$
	\end{fact}
	In Lemma \ref{L: restriction of nu to subset} we will prove an analogous version for the other unstable distinguished sorts.

	We can now recall the construction of the infinitesimal subgroups for unstable $D$, phrased for $D$-balanced groups (instead of $D$-groups) and $D$-critical sets (instead of $D$-sets): 
	
	\begin{fact}\cite[Proposition 5.11]{HaHaPeGps}\label{F: properties of nu}
		Let $D$ be one of the unstable distinguished sorts and let $G$ be a definable $D$-balanced group, with $\drk(G)>0$.
		\begin{enumerate}
			\item If $X\subseteq G$ is $D$-critical, all definable over $A$, then for every $A$-generic $a,b\in X$, the set $\nu^A_X(a)a^{-1}$ is a (type-definable) subgroup of $G$ and  
			$\nu^A_X(a)a^{-1}\equiv\nu^A_X(b)b^{-1}\equiv a^{-1}\nu^A_X(a)$. We denote this group $\nu^A_X$.
			
			\item If $X,Y\sub G$ are $D$-critical, definable over $A$ and $B$, respectively, then $\nu^A_X\equiv\nu^B_Y$, and we denote this type-definable subgroup  $\nu_D(G)$ (or just $\nu_D$), \emph{the infinitesimal type-definable subgroup of $G$ with respect to $D$}.
			\item For every $g\in G(\CK)$, we have $g\nu_D(G) g^{-1}\equiv \nu_D(G)$.
			\item $\dpr(\nu_D(G))=\drk(G)=\adrk(G)$. 
		\end{enumerate}
	\end{fact}
	
	For each of the different unstable distinguished sorts we will offer below several alternative descriptions of the subgroups $\nu_D$. First, recall: 
	
	\begin{fact}\label{F: concrete description of nu for D-balanced}
		Let $D$ be an unstable distinguished sort and $G$ a definable $D$-balanced group with $\drk(G)>0$. 
		\begin{enumerate}
			\item \cite[Section 4.2]{HaHaPeSemisimple} Assume that $D=K$ or $D=\bk$. For every $D$-critical set $X\subseteq G$ and every $f:X\to D^n$ a definable injection, with $n=\dpr(X)$,  and $c\in X$ generic over all the data, \[\nu_D\equiv \{f^{-1}(U)c^{-1}: U\subseteq D^n \text{ definable open containing $f(c)$}\}.\] 
			\item \cite[Section 5.1, Proposition 5.2]{HaHaPeSemisimple} Assume that $D=K/\CO$. Then, there exists  a definable subgroup $G_0\leq G$ and a definable isomorphism $f:G_0\to B$, where $B\leq (K/\CO)^n$ is an open ball around $0$ with $n=\drk(G)$. Moreover, for any such $G_0, f$ and $B$, we have \[\nu_{K/\CO}\equiv\{f^{-1}(U): U\leq B \text{ is an open ball around $0$}\}.\] 
			\item \cite[Lemma 5.6]{HaHaPeSemisimple} Assume that $D=\Gamma$. The type-definable group $\nu_\Gamma$ is definably isomorphic to  $\{U: U\subseteq \Gamma^n \text{ a $\Gamma$-box around $0$}\}$, where $n=\adrk(G)$.
		\end{enumerate}
	\end{fact}
	
	We shall use the following simple observations in the case that $G=(D^n,+),$ for unstable $D$:
	\begin{remark}\label{R: how nu looks}
		
		\begin{enumerate}
			\item For $D=K$ or $D=\bk$, \[\nu_D(D^n,+)\equiv \{U: U\subseteq D^n \text{ definable open neighborhood of $0$}\}.\]
			\item $\nu_{K/\CO}((K/\CO)^n,+)\equiv \{U\subseteq (K/\CO)^n: U\leq B \text{ is an open ball around $0$}\}$. 
			
			In particular, when $\CK$ is $p$-adically closed, $\nu_{K/\CO}((K/\CO)^n,+)$ contains all the torsion points of $((K/\CO)^n,+)$ (see \cite[Fact 7.3(3)]{HaHaPeSemisimple}). 
			\item $\nu_\Gamma(\Gamma^n,+)\equiv \{U: U\subseteq \Gamma^n \text{ a $\Gamma$-box around $0$}\}$.
		\end{enumerate}
	\end{remark}

	We can now state the alternative characterizations of $\nu_D$ referred to above. Those are given in terms of $D$-critical sets. 
	\begin{proposition}\label{P: nu withough generics}
		Let $D$ be an unstable distinguished sort and $G$ be an $A$-definable $D$-balanced group with $\drk(G)>0$. Then the following types are logically equivalent.
		\begin{enumerate}
			
			\item (The definition of $\nu_D$) Given a $D$-critical set $Z_1\sub G$ over $A$ and $d_1\in Z_1$ $A$-generic, $\rho_1=\{Xd_1^{-1}:X\in \nu_{Z_1}(d_1)\}$.

			\item $\rho_2=\{XX^{-1}: X\subseteq G \text{ is $D$-critical}\}$.
			
			\item Given a $D$-critical $Z_2\subseteq G$, $\rho_3=\{XX^{-1}: X\subseteq Z_2 \text{ is $D$-critical}\}$.
			
			\item Given a  $D$-critical $Z_3\subseteq G$ over $A$ and an  $A$-generic $d_3\in Z_3$,  $\rho_4=\{UU^{-1}:U\in \nu_{Z_3}(d_3)\}$.

		\end{enumerate}
	\end{proposition}
	
	\begin{proof} 
		We show that $\rho_1\vdash \rho_2\vdash \rho_3\vdash \rho_4 \vdash \rho_1$. We shall use the fact  that the definition of $\nu_D$ does not depend on the choice of the $D$-critical set and its generic point (Fact \ref{F: properties of nu}). 
		
		We first prove:
		\begin{claim} 
			For any $D$-critical $Z\sub G$ and an $A$-generic $d\in Z$ (for $A$ over which everything is defined), if $U\in \nu_Z(d)$ then there exists $X\in \nu_Z(d)$ such that $XX^{-1}\sub Ud^{-1}$. 
		\end{claim}
		\begin{claimproof}
			Since $\nu=\nu_D$ is a subgroup, $\nu\cdot\nu^{-1}$ is logically equivalent to $\nu$, so there exist, by compactness,  $V_1,V_2\in \nu_Z(d)$ such that $V_1d^{-1}\cdot (V_2d^{-1})=V_1V_2^{-1}\sub Ud^{-1}$. As $\nu_Z(d)$ is a filter base (\cite[Lemma 5.6]{HaHaPeGps}), there exists $X\in \nu_Z(d)$ (hence $D$-critical) such that $X\subseteq V_1\cap V_2$ so $XX^{-1}\subseteq Ud^{-1}$.
		\end{claimproof}
		We can now prove the proposition:
		
		$\rho_1\vdash \rho_2$: Let $X\sub G$ be any $D$-critical set and then fix in it a generic $d_1\in X$. Then $X\in \nu_X(d_1)$, so $\nu_D=\nu_X(d_1)d_1^{-1}\vdash XX^{-1}$, hence $\nu_D=\rho_1\vdash \rho_2$.
		
		\vspace{.2cm} $\rho_2\vdash \rho_3$: This is clear.
		
		\vspace{.2cm} $\rho_3\vdash \rho_4$: Fix $D$-critical sets $Z_2,Z_3\subseteq G$ and $d_2\in Z_2$, $d_3\in Z_3$ generic elements.
		
		Given $U\in \nu_{Z_3}(d_3)$, it follows from the claim that there exists $X\sub \nu_{Z_2}(d_2)$ (in particular $X\subseteq Z_2$ is $D$-critical), such that $XX^{-1}\sub UU^{-1}$. Thus, $\rho_3\vdash \rho_4$.
		
		\vspace{.2cm} $\rho_4\vdash \rho_1$: This follows easily from the claim.
	\end{proof}

	\begin{remark}
		The above lemma as well as its proof remain true in the vicinic setting of \cite[\S4]{HaHaPeGps} after systematically replacing all occurrences of ``$D$-critical'' sets with ``$D$-sets''. 
	\end{remark}

	We now extend the above to the case when $D$ is stable, namely $D=\bk$ in  the $V$-minimal case. By \cite[Proposition 6.2]{HaHaPeGps}, if $G$ is locally strongly internal to the residue field $\bk$ then there exists a connected $\omega$-stable definable subgroup of $G$,  also denoted $\nu_\bk$,  which is strongly internal to $\bk$ and whose dp-rank is the $\bk$-rank of $G$. The previous proposition can now be extended to cover the stable case in the following sense:

	\begin{lemma}\label{L:nu in XX-1}
		Let $G$ be a definable group in $\CK$ and $D$ any distinguished sort. If  $G$ is $D$-balanced with $\drk(G)>0$ then $\nu_D$ is equivalent to the partial type $\{XX^{-1}: X\subseteq G \text{ is $D$-critical}\}.$

	\end{lemma}
	\begin{proof}
		For $D$ unstable this is Proposition \ref{P: nu withough generics} $(1)\Leftrightarrow (2)$, so assume that $D=\bk$ is stable, and $\CK$ is V-minimal. Since $\nu_\bk$ is in that case a definable $\bk$-critical set and $\nu_\bk\nu_\bk^{-1}=\nu_\bk$ it will suffice to prove that for every $\bk$-critical set $X$, $\nu_{\bk}\vdash XX^{-1}$. 
		
		Let $G_0=\nu_{\bk}$, as provided by \cite[Theorem 7.8]{HaHaPeGps}. It is a connected definable normal subgroup of $G$ of finite Morley rank, satisfying $\mr(G_0)=\mbox{\bk-rk(G)}$ (see Proposition 6.2 \emph{loc. cit.} for more details). 
		
		Let $X\sub G$ be a $\bk$-critical set. We have to show that $G_0\sub XX^{-1}$, so there is no harm assuming $\mathrm{MD}(X)=1$. Let $G_1\sub G$ be the stabilizer of the (unique) generic type of $X$ (this is well-defined since $X$ is stable and stably embedded and its unique generic type is, therefore, definable. See also the proof of \cite[Proposition 6.2]{HaHaPeGps} for a more detailed argument). So $G_1\sub XX^{-1}$ and by \cite[Claim 6.2.1]{HaHaPeGps} $\mr(G_1)=\mr(X)=\mr(G_0)$.
		
		The group $G_0G_1$ is $\bk$-critical by \cite[Lemma 6.1]{HaHaPeGps}, thus $\mr(G_0G_1)=\mr(G_0)$, so (as $G_0$ is connected) $G_0\sub G_1\sub XX^{-1}$, as claimed. 	
	\end{proof}
	
	\begin{remark}
		The description of $\nu_D$ given above, not present in our previous papers, was motivated by Johnson's description of the infinitesimals in \cite{JohnDpJML}.
	\end{remark}

	The minimality of $\nu_D$ in the appropriate sense  was proved in \cite[Corollary 5.12]{HaHaPeGps} for unstable $D$ and we now prove it for all $D$.
	
	\begin{lemma}\label{L: nu is smallest}
		Let $G$ be a $D$-balanced group definable in $\CK$, with $\drk(G)>0$. For any type-definable subgroup $\mu$ of $G$ strongly internal to $D$ with $\dpr(\mu)=\dpr(\nu_D)$, we have $\nu_D\vdash \mu$. 
	\end{lemma}
	\begin{proof}
		Let $\mu\vdash Y$, since $\mu$ is a type-definable group there is a definable $X$, $\mu\vdash X$, with $XX^{-1}\subseteq Y$. By the assumptions, we can choose $X$ to be $D$-critical. By Lemma \ref{L:nu in XX-1}, $\nu_D\vdash XX^{-1}$, so $\nu_D\vdash Y$. It follows that $\nu_D\vdash \mu$. 
	\end{proof}

	\begin{lemma}\label{L:nu divisible}
		Let $G$ be a $D$-balanced group definable in $\CK$ where $D$ is unstable, and $\drk(G)>0$. Then
		\begin{enumerate}
			\item For all $D$, $\nu_D$ is divisible (i.e. radicable\footnote{I.e., for any $a\in \nu_D(\widehat \CK)$ and integer $n\neq 0$ there exists $b\in \nu_D(\widehat \CK)$ for which $b^n=a$.}).
			\item If $D\neq \KOp$, then $\nu_D$ is torsion-free.
			\item If $D=\Gamma$ then there is a definable finite index subgroup $G_1\sub G$ such that $\nu_\Gamma\vdash Z(G_1)$. In particular, $\nu_\Gamma$ is commutative.
			\item If $D=K/\CO$ then $\nu_D$ is commutative. 
		\end{enumerate}
	\end{lemma}
	\begin{proof} We prove all clauses for each of the relevant sorts $D$.

		Assume that $D=\Gamma$. By Fact \ref{F: concrete description of nu for D-balanced},  $\nu_D$ is definably isomorphic to $\{U:U\subseteq \Gamma^n \text{ a $\Gamma$-box around $0$}\}$. If $\CK$ is not $p$-adically closed then $\Gamma$ is divisible, and it is easy to see that $\nu_D$ is divisible as well (if $(a,b)$ is an open interval containing $0$ then so is $(\frac{1}{n}a,\frac{1}{n}b)$ for any natural number $n$). Assuming that $\CK$ is $p$-adically closed, every $\Gamma$-box around $0$ is a cartesian product of sets of the form $(b_i,c_i)\cap\{x_i:x_i\equiv 0 \mod (m_i)\}$, where $m_i$ ranges over all integers. In particular, $\nu_\Gamma$ is divisible. It follows, in addition, that in all these cases  $\nu_\Gamma$ is torsion-free and commutative, since $\Gamma$ is. Clause (3) is just \cite[Proposition 6.1]{HaHaPeSemisimple}.
		
		Assume that $D=K$ (in all settings) or $D=\bk$ in the power-bounded $T$-convex setting. In that case, the function $x\mapsto x^n$ is differentiable with respect to the underlying field structure and its differential at the identity is invertible (see, e.g.,  the proof of  \cite[Lemma 7.1]{HaHaPeGps}). By the inverse mapping theorem there is a definable neighborhood $U$ of $e$ such that the restriction of $x\mapsto x^n$ is an open map injective on $U$. By the definition of $\nu_D$, this implies that $x\mapsto x^n$ is surjective on $\nu_D$, so $\nu_D$ is divisible,  and that $\nu_D$ is torsion-free.  
		
		Assume that $D=K/\CO$.  By Fact \ref{F: concrete description of nu for D-balanced} (and Remark \ref{R: how nu looks}), there exists a definable group isomorphism between  $\nu_D$ and   $\nu_{K/\CO}((K/\CO)^m,+)= \{U\subseteq (K/\CO)^m: U\leq B \text{ is an open ball around $0$}\}$ for some $m$. Thus, it is sufficient to show that the latter group is divisible.
		
		Let $a\in \nu_{K/\CO}((K/\CO)^m,+)$, $n\in \mathbb N$ and let $B_{>r}(0)$ be a ball in $(K/\CO)^m$ around $0$ of radius $r<0$ (if $\CK$ is $p$-adically closed, then $r<\mathbb{Z}$). Thus, $a\in B_{>r+v(n)}$ and since $K$ is divisible, there is some $b\in (K/\CO)^m$ with $nb=a$ and $v(b)=v(a)-v(n)>r$ so $b\in B_{>r}(0)$. We conclude that $b\in \nu_{K/\CO}((K/\CO)^m,+)$ as well. Ranging over all $n$  we get that $\nu_D$ is divisible. As $K/\CO$ is commutative, so is $\nu_D$.  When $\CK$ is not $p$-adically closed, $K/\CO$ is torsion-free, hence so is $\nu_D$.
	\end{proof}
	
	\begin{remark}
		If $D=\bk$ in the $V$-minimal case  then for every $n$, $\dpr((\nu_D)^{(n) })=\dpr(\nu_D)$ (the same proof as in Claim \ref{C:dim of G^n}). However, $\bk$-algebraic groups need not be divisible, and could have torsion.
	\end{remark}
	
	\subsection{Infinitesimal subgroups for all}\label{ss:infin for all}
	We now extend the definition of  infinitesimal subgroups to all groups definable in $\CK$, not necessarily $D$-balanced. In particular, we define the infinitesimal subgroups of definable groups, which are locally {\em almost} strongly internal to $D$ but may not be   locally strongly internal to $D$. As we have already seen, groups that are not $D$-balanced can only arise when $D=\KOp$. Recall that $H_G$ is the unique minimal normal finite subgroup of $G$ such that $G/H_G$ is $D$-balanced, and if $D\neq \KOp$ then $H_G=\{e\}$.
	
	\begin{definition}
		Let $G$ be a group definable in $\CK$ and assume that $G$ is locally almost strongly internal to a distinguished sort $D$. 
		Then define the \emph{$D$-infinitesimal subgroup of $G$} as \[\nu_D(G):=\pi^{-1}(\nu_D(G/H_G))\] where $\pi: G\to G/H_G$ is the quotient map.  
	\end{definition}
	
	\begin{notation}
		For the sake of uniformity of notation, we let $\nu_D(G):=\{e\}$ if $G$ is not locally almost strongly internal to $D$.  
	\end{notation}
	
	\begin{remark}
		\begin{enumerate}
			\item If $G$ is $D$-balanced and $\drk(G)>0$ then $\nu_D$ as defined above is the same $D$-infinitesimal subgroup of $G$ provided by \cite[Proposition 5.11]{HaHaPeGps}. 
			\item In \cite{HaHaPeGps} we have not defined the $D$-infinitesimal subgroups of groups that are not $D$-groups, (e.g., when $D$ is stable or $G$ is not locally strongly internal to $D$). In that regard, the current definition extends the original one. 
			\item For groups locally strongly internal to $D$ but not $D$-balanced, the current definition of $\nu_D$ may differ from the one in \cite{HaHaPeGps}. Indeed, for $G$  as in Examples \ref{E:not balanced}, \ref{E:group but not balanced} the group $G\times K/\CO$ is locally  strongly internal to $D=K/\CO$ but is not $D$-balanced. Its almost $D$-rank is $2$, implying that $\dpr(\nu_D(G\times K/\CO))=2$.  On the other hand, the $D$-rank of $G\times K/\CO$ is $1$. Since it is locally strongly internal to $K/\CO$, the construction of \cite[Proposition 5.11]{HaHaPeGps} can be carried out, giving rise to an infinitesimal subgroup of rank $1$ which will be a subgroup of our $\nu_D$. 
		\end{enumerate}
	\end{remark}

	We aim to prove a characterization of $\nu_D(G)$ generalizing Lemma \ref{L:nu in XX-1}. Towards that end, we first prove:
	
	\begin{lemma}\label{L:some results on almost D-crit vs D-crit}
		Let $G$ be a definable group locally almost strongly internal to a distinguished sort $D$. 
		Let $f:G\to G/H_G$ be the quotient map. For any definable subset $X\subseteq G$ almost strongly internal to $D$ with $\dpr(X)=\adrk(G)$,  there exists an almost $D$-critical definable subset $X'\subseteq X$ such that $f(X')$ is $D$-critical in $G/H_G$ and $f^{-1}(f(X'))=X'$.
	\end{lemma}
	\begin{proof}
		The result is trivially true if $G$ is $D$-balanced so we assume $D=\KOp$ and in particular $D$ is unstable.
		
		Let $f:G\to G/H_G$ be the quotient map. By Proposition \ref{Canonical H}, $|H_G|=m_{crit}$.  Let $X\subseteq G$ be almost strongly internal to $D$ with $\dpr(X)=\adrk(G)$.
		By \cite[Proposition 4.35(1)]{HaHaPeGps}, there is a definable $X_1\sub X$,
		such that $f(X_1)$ is $D$-critical in $G/H_G$. 
		
		Then $\dpr(X_1)=\adrk(G)$, and by the minimality of $m_{crit}$, a generic fiber of $f\rest X_1$ has size $m_{crit}$. Thus, there exists a definable $X'\sub X_1$ such that the fibers of $f\rest X'$ have size $|H_G|$ i.e. $f^{-1}(f(X'))=X'$. Clearly, $f(X')$ is also $D$-critical in $G/H_G$.
	\end{proof}
	
	\begin{proposition}\label{P:characterizatio of nu_D in general}
		Let $G$ be a definable group locally almost strongly internal to a distinguished sort $D$. 
		
		\begin{enumerate}
			\item The following partial types are logically equivalent:
			\begin{enumerate}
				\item[(a)] $\nu_D$
				\item[(b)] $\{XX^{-1}: X\subseteq G  \text{ is almost $D$-critical}\}$.
				
				\item[(c)] $\{XX^{-1}: X\subseteq G  \text{ is almost strongly internal to $D$ and $\dpr(X)=\adrk(G)$}\}$.
				
				\item[(d)] For any $Y\subseteq G$ almost $D$-critical, $\{XX^{-1}: X\subseteq Y  \text{ is almost $D$-critical}\}.$

			\end{enumerate}

			\item Let $\mu$ be a type-definable group concentrated on a definable set that is almost strongly internal to $D$. If $\dpr(\mu)=\adrk(G)$ then $\nu_D\vdash \mu$. 
		\end{enumerate}
	\end{proposition}
	\begin{proof}
		Recall that $X\sub G$ is almost $D$-critical if $\dpr(X)=\adrk(G)$ and there is $f:X\to D^n$, $n=\dpr(X)$, with the fibers of size $m_{crit}$. 
		
		(1)     
		If $G$ is $D$-balanced then $X\sub G$ is $D$-critical if and only if it is almost $D$-critical, so Proposition \ref{P: nu withough generics} implies the result. 
		
		If $G$ is not $D$-balanced let $f:G\to G/H_G$ be the quotient map and then $\nu_D(G/H_G)$ is logically equivalent to $\{YY^{-1}: Y\subseteq G/H_G \text{ is $D$-critical} \}$  (Proposition \ref{P: nu withough generics} again).

		Thus, $\nu_D(G)$ is logically equivalent to $$\{f^{-1}(Y)f^{-1}(Y)^{-1}: Y \subseteq G/H_G \text{ is $D$-critical in } G/H_G\}.$$ By definition of $H_G$, every such $f^{-1}(Y)$ is almost $D$-critical. Applying Lemma \ref{L:some results on almost D-crit vs D-crit}, one can prove the equivalence to (b). The equivalence to (c) and (d) is similarly proved using Proposition  \ref{P: nu withough generics}.
		
		(2) By the rank assumption, $\mu$ is equivalent to a collection of sets almost strongly internal to $D$ of dp-rank $\adrk(G)$. Since $\mu$ is a subgroup we know that $\mu\equiv \mu\mu^{-1}$, so the conclusion follows from (1c) above. 
	\end{proof}
	
	For the following see also \cite[Lemma 2.16]{HaHaPeSemisimple}.
	
	\begin{proposition}\label{P:general nu of finite index subgroup}
		Let $D$ be a distinguished sort and let $G$ be a  definable group in $\CK$ locally almost strongly internal to $D$. 
		
		For any definable subgroup $G_1\leq G$, if $\adrk(G_1)=\adrk(G)$ then  $\nu_D(G_1)\equiv \nu_D(G)$. In particular, this is the case if $G_1$ has finite index in $G$.
		
		Consequently,  if $H\leq G$ is a finite normal subgroup then $\nu_D(G)\vdash C_G(H)$.
	\end{proposition}
	\begin{proof}
        By assumption, every almost strongly internal subset of $G_1$ is also such for $G$. Applying Proposition \ref{P:characterizatio of nu_D in general}(2) to both groups, yields $\nu_D(G_1)\equiv \nu_D(G)$. If $G_1$ has finite index in $G$ then clearly $\adrk(G_1)=\adrk(G)$.
		
		The last statement follows by applying the above to the finite index subgroup $C_G(H)$.
	\end{proof}

	The following complements Proposition \ref{P:general nu of finite index subgroup}. It also shows that $\nu_D(G)$ is invariant under definable isomorphisms of $G$. 
	
	\begin{proposition}\label{P:nu under finite to one}
		Let $D$ be a distinguished sort and $G_1, G_2$ be definable groups in $\CK$ locally almost strongly internal to $D$.
		
		Let $f:G_1\to G_2$ be a definable surjective homomorphism with finite kernel. Then $f_*\nu_D(G_1)\equiv \nu_{D}(G_2)$. In particular, if $f$ is an isomorphism then it induces an isomorphism of $\nu_D(G_1)$ with $\nu_D(G_2)$.
		
		If $D$ is unstable and different from $\KOp$ then $f$ induces an isomorphism between $\nu_D(G_1)$ and $\nu_D(G_2)$.
	\end{proposition}
	\begin{proof} 
		By Fact \ref{F:finite index same adrk}, $\adrk(G_1)=\adrk(G_2)$. If $Y\sub G_2$ is almost strongly internal to $D$ and $\dpr(Y)=\adrk(G_2)$ then  $f^{-1}(Y)\sub G_1$ is almost strongly internal to $D$, with $\dpr(f^{-1}(Y))=\adrk(G_1)$. This shows, using Proposition \ref{P:characterizatio of nu_D in general}(1c), that $f_*(\nu_D(G_1))\vdash \nu_D(G_2)$.

		If $X\subseteq G_1$ is a definable subset almost strongly internal to $D$ with $\dpr(X)=\adrk(G_1)$ then $f(X)$ contains a subset of full dp-rank which is almost strongly internal to $D$ (\cite[Lemma 2.9(2), Lemma 3.9, Lemma 4.3]{HaHaPeGps}). This shows, using Proposition \ref{P:characterizatio of nu_D in general}(1c), that $\nu_D(G_2)\vdash f_*\nu_D(G_1)$.   

		Now, assume that $D$ is unstable and different from $\KOp$. By Lemma \ref{L:nu divisible}(2), $\nu_D(G_1)$ is torsion-free so $f\restriction \nu_D(G_1)$ is injective; the result follows.
	\end{proof}
	
	Let us note that the above does not extend to  definable homomorphisms with infinite kernel. In the example below we have two groups $G_1, G_2$ that are $D$-balanced and of  the same $D$-ranks for all $D$, $\dpr(G_1)=\dpr(G_2)$ and a definable function $f:G_1\to G_2$; however, $f$ does not map $\nu_D(G_1)$ onto $\nu_D(G_2)$.
	
	\begin{example}\label{E: homomorphisms}
		Let us assume that $\CK$ is $V$-minimal (a similar example will work in the other cases). Let $G_1=(K/\m)\times K/\CO$. The $K/\CO$-rank and $\bk$-rank of $G_1$ are both $1$. Fix $B\supseteq \CO$ an open ball, and consider  $f: G_1\to G_2= K/\CO\times K/B$ given by $(x,y)\mapsto (x+\CO, y+B)$. Notice that $K/B\cong K/\m$.  For $D=\bk$ we have $\nu_D(G_1)= (\CO/\m)\times \{0\}\cong (\bk,+)$ and $\nu_D(G_2)=\{0\}\times \bar B/B\cong (\bk,+)$ (with $\bar B$ the closed ball corresponding to $B$). 
		However, unlike the case of finite kernel, we have $f_*(\nu_D(G_1))=\{(0,0)\}$. 
		
		An analogous argument shows that $f_*(\nu_{K/\CO})(G_1)=\{(0,0)\}$.   
	\end{example}

	\subsection{The case of $\KOp$}
	
	In this section we prove the main properties of $\nu_D(G)$ for $D=\KOp$.

	\begin{proposition}\label{P:lasi in K/O p-adic.1}
		Let $\CK$ be $p$-adically closed and $G$ a definable group. For $D=K/\CO$, the following are equivalent:
		
		\begin{enumerate}
			\item $G$ is locally almost strongly internal to $D$ and $\adrk(G)\geq m$.
			\item There exists a definable subgroup of $G$ which is definably isomorphic to a group $B/C$, where $B\leq (K/\CO)^m$ is a ball around $0$ and $C\leq  B$ is a finite subgroup.
		\end{enumerate}
	\end{proposition}
	\begin{proof}
		Assume (2) and let $r=\min\{v(x):x\in C\}$. Since $C$ is finite, $r\in \mathbb Z$ (\cite[Fact 7.3]{HaHaPeSemisimple}) and  
		$B_{\geq r}=\{x\in (K/\CO)^m:v(x)\geq r\}$ is a finite group containing $C$.  The map $x+C\mapsto x+B_{\geq r}$ is finite-to-one mapping $B/C$ into $(K/\CO)^m/B_{\geq r}$. Since $(K/\CO)^m/B_{\geq r}$ is definably isomorphic to $(K/\CO)^m$ it follows that a definable subgroup of $G$ of dp-rank $m$ is almost strongly internal to $K/\CO$, so  $\adrk(G)\geq \dpr(B)\geq m$.
		
		We now prove the direction (1) $\Rightarrow$ (2).  As a first step, we replace $G$ by its finite index subgroup $C_G(H_G)$ and assume that $H_G$ is central. This does not affect the almost $D$-rank.
		
		Let $f:G\to G/H_G$ be the quotient map. 
		
		\begin{claim}
			There is a definable abelian subgroup $G_0\leq G$ containing $H_G$, such that $f(G_0)$ is definably isomorphic to a ball $B\leq (K/\CO)^m$.
		\end{claim}
		\begin{claimproof}
			Since $\adrk(G)\geq m$ then, by Fact \ref{F: concrete description of nu for D-balanced}, $G/H_G$ has a definable subgroup $N_1$ which is definably isomorphic to a ball in $(K/\CO)^m$, so in particular $N_1$ is abelian of dp-rank $m$. For $G_1=f^{-1}(N_1)$, the map $f\rest G_1:G_1 \to N_1$ has finite central kernel and abelian image. Thus, by  \cite[Lemma 2.21(2)]{HaHaPeSemisimple},    $Z(G_1)$ has finite index in $G_1$ so $f(Z(G_1))\sub N_1$ has dp-rank $m$. By Fact \ref{F: concrete description of nu for D-balanced}, it contains a group $N_2$ definably isomorphic to a ball in  $(K/\CO)^m$.
			
			Let $G_0=f^{-1}(N_2)\cap Z(G_1)$. It is abelian, and its image $N_2$ is definably isomorphic to a ball $B\sub (K/\CO)^m$. Note that $f^{-1}(N_2)$ and $G_1$ contain $H_G$ by definition, so as $H_G$ is central in $G$, it is contained in $G_0$.  This proves the claim.
		\end{claimproof}
		
		Let $\widehat f:G_0\to B$ be the composition of (the restriction of) $f$ with the isomorphism to a ball in $(K/\CO)^m$. 
		Now, if $n=|H_G|$, then there is a definable function $\beta: B \to G_0$, given by $\beta(x)=f^{-1}(x)^n$. It is well-defined since $|\ker f|=n$ and because $G_0$ is abelian it is a group homomorphism. Note that $C:=\ker(\beta)$ is contained in  $\{x\in B:nx=0\}$ and by \cite[Fact 7.3]{HaHaPeSemisimple} it is finite. 
		Thus, $\dpr(\beta(B))=m$ and $\beta(B)\sub G$ is definably isomorphic to $B/C$. 
	\end{proof}

	\begin{proposition}\label{P:lasi in K/O p-adic}
		Let  $\CK$ be $p$-adically closed, $G$ be a definable group, $D=\KOp$ and $\adrk(G)=m>0$. Then, 
		\begin{enumerate}
			\item[(i)] $\nu_{D}(G)$ is definably isomorphic to the quotient of $\nu_D((K/\CO)^m,+)$ by a finite subgroup $C$. 
			\item[(ii)] $\nu_{D}(G)$ is divisible. 
		\end{enumerate}  
	\end{proposition}
	\begin{proof} 
		(i) Let $m=\adrk(G)$. By Proposition \ref{P:lasi in K/O p-adic.1}, there exists a definable subgroup $G_1\leq G$, $\dpr(G_1)=m$, which is definably isomorphic to $B/C$, for $B\sub (K/\CO)^m$ a ball and $C$ finite a finite subgroup. By Proposition \ref{P:general nu of finite index subgroup}, $\nu_D(G)\equiv \nu_D(G_1)$; the latter is definably isomorphic to $\nu_D(B/C)$. By Proposition \ref{P:nu under finite to one}, $\nu_D(B/C)$ is the image of $\nu_D(B)$ under the quotient map by $C$. Since $\nu_D(B)\equiv \nu_D((K/\CO)^m,+)$, the result follows. 
		
		(ii) Because  $\nu_D((K/\CO)^m,+)$ is divisible, so is its quotient by $C$.
	\end{proof}

	As a consequence of Proposition \ref{P:lasi in K/O p-adic} and Proposition \ref{P:lasi=lsi}, Lemma \ref{L:nu divisible} holds for all definable groups:
	\begin{corollary}\label{C: nu divisible}
		Let $G$ be an infinite group definable in $\CK$ and locally almost strongly internal to an unstable distinguished sort $D$. Then
		\begin{enumerate}
			\item $\nu_D$ is divisible.
			\item If $D\neq \KOp$, then $\nu_D$ is torsion-free.
			\item If $D$ is either $\Gamma$ or $K/\CO$ then $\nu_D$ is commutative. 
		\end{enumerate}
	\end{corollary}

	\section{The infinitesimal groups commute}
	
	Our goal in this section is to prove that the infinitesimal subgroups associated with any two distinct distinguished sorts commute. In the following section this is used to prove that the subgroup generated by the four subgroups is isomorphic to their direct product. We first need some additional technical preparations. 
	
	\subsection{Fine local analysis}
	
	Local analysis in, e.g., o-minimal structures, builds heavily on the fact that given a definable set $X$ and a generic point $x\in X$ there are arbitrarily small generic neighborhoods of $x$ in $X$ (i.e. neighborhoods that are generic vicinities), over any set of parameters. Our definition of vicinic sets in \cite[\S 4]{HaHaPeGps} aimed to capture this property in a more general setting, covering all the distinguished sorts. In the present section, we extend this analysis by allowing parameters coming from different distinguished sorts.  
	
	The following proof requires  specific properties of the different distinguished sorts.  We refer the reader to \cite{HaHaPeVF,HaHaPeGps} for a more detailed exposition covering the main properties of these structures. Throughout the proof, we give specific references for those  properties we need.

	\begin{proposition}\label{P: topological pair 2}
		Let $D_1,D_2$ be distinguished sorts with $D_2$ unstable.  	
		Let $a\in D_1^m$ and let $b$ be generic in $D_2^n$ over some parameter set $A$.  Then for every  $U\in \nu_{D_2^n}(b)$   
		there exists $V\in \nu_{D_2^n}(b)$, $V\subseteq U$,  such that $\dpr(a/Ab[V])=\dpr(a/Ab)$.
	\end{proposition}
	\begin{proof}  
		Fix $U\in \nu_{D_2^n}(b)$. We first observe that  it is enough to find a set $U'$, $b\in U'\subseteq U$, with $U'$ definable over some $B\supseteq A$,  such that $\dpr(a/Bb)=\dpr(a/Ab)$, and a subset $U''\subseteq U'$ definable over some $C\supseteq A$, such that  $b$  is generic in $U''$ over $C$.
		
		Indeed, if we can find such sets, then let $p$ be a complete type over $BCb$ extending $\tp(a/Bb)$ such that $\dpr(p)=\dpr(a/Bb)$. Let $a'\models p$ and let $\sigma$ be an automorphism over $Bb$ mapping $a'$ to $a$. So $\sigma(U'')$ is a $\sigma(C)$-definable subset of $U'$ containing $b$.  By our choice of $p$, $\dpr(a/Bb)=\dpr(a'/BbC)=\dpr(a/Bb\sigma(C))$ (so $\dpr(a/Ab)=\dpr(a/Ab\sigma(C))$). 
		Also, $\dpr(b/A)=\dpr(b/C)=\dpr(b/\sigma(C))$, so setting $V=\sigma(U'')$ (defined over $\sigma(C)$), we get that  $V\in \nu_{D_2^m}(b)$, and $\dpr(a/Ab[V])=\dpr(a/Ab)$, as required. 
		
		\vspace{.2cm}
		
		Thus, let us see that we can find $U'$ and $U''$ as above. In the classes of valued fields we consider, $D_2$ can be  either an SW-uniformity, $\Gamma$ in the $p$-adically closed case (so a $\mathbb{Z}$-group) or $\KOp$. We investigate each case separately.  \\

			\noindent    \textbf{Case 1: $D_2$ is an SW-uniformity.}
			
			By Fact \ref{F: nu in SW for general Y}, we may replace $U$ by an open definable subset of $D_2^n$  containing $b$.
			
			Let $\{U_t[b]\}_{t\in T}$ be a definable neighborhood base at $b$ (so $U_t[b]$ is definable over $bt$), and shrinking $U$ if needed,  assume that $U=U_{t_0}[b]$ for some $t_0\in T$. For $t_1, t_2\in T$ write $t_1\leq t_2$ if $U_{t_1}[b]\subseteq U_{t_2}[b]$. Note that, since the topology is Hausdorff without isolated points, $(T,\leq)$ has no minimal elements. 
			
			The proof is, in essence, the same as that of \cite[Lemma 3.11(1)]{HaHaPeVF}, but minor modifications are needed for the first part of the proof since $a$ need not belong to an SW-uniformity. In \cite{HaHaPeVF} the fact that $D_1$ is an SW-uniformity was applied to obtain definability of dp-rank. Instead, since $a$ belongs to a distinguished sort $D_1$ and dp-rank in $D_1$ is algebraic (i.e., given by $\dim_{\acl}$), we get upward type-definability of the dp-rank. Specifically, this means that for any formula $\varphi(x,y)$ and integer $m$, if $x$ is a tuple from $D_1$ then the set
			$\{c:\dpr(\varphi(D_1,c))\geq m\}$ is type-definable over the same parameters as $\phi$ (see also \cite[Corollary 3.5]{Simdp}, although Simon does not explicitly allow parameters from outside the dp-minimal sort).
			
			In view of the above, the set $P(s)$ of $s\in T$ with $s\le t_0$ and $\dpr(a/Abs)\geq \dpr(a/Ab)$ is type-definable. If $P(s)$ is consistent, then any realization $s_0$ of $P(s)$ will be sufficient and we set $U'=U_{s_0}[b]$, definable over $B:=As_0$. The proof  that $P(s)$ is indeed consistent goes through verbatim as in \cite[Lemma 3.11(1)]{HaHaPeVF}.

			This gives the desired definable set $U'$ and parameter set $B$. In this situation, $U'$ is an open set containing $b$ so \cite[Proposition 3.12]{HaHaPeVF} applies, and we find the desired $U''$ over a suitable $C$. \\
			
			\noindent \textbf{Case 2: $D_2=\KOp$.}
			
			By \cite[Lemma 3.6]{HaHaPeGps}, there exists $U_0=B_{\geq t}(b)\subseteq U$ for some $t<\Zz$. Consider the collection of formulas in the $\Gamma$-variable $r$ 
			
			$p(r):=\text{``}\dpr(a/Abr)=\dpr(a/Ab)"\cup \text{``}t<r<\Zz".$ 
			Note that the relation (in the variable $r$) $\dpr(a/Abr)\geq \dpr(a/Ab)$ is type-definable since $D_1$ has algebraic dp-rank. To show consistency, it suffices to show that for all $k\in \Nn$ the type: 
			\[
			p_k(r):=\text{``}\dpr(a/Abr)=\dpr(a/Ab)"\cup \text{``}t<r<-k"
			\]
			is consistent. But since $t<\Zz$ e.g., $r=-k-1$ satisfies this type. So $p_k(r)$ is consistent, implying that $p(r)$ is consistent, as needed. 
			Let $r$ be any realization of $p(r)$; so $b\in B_{\geq r}(b)\subseteq U$, $B_{\geq t}(b)$ is definable over $bt$,  and $\dpr(a/Abr)=\dpr(a/Ab)$. Set $U'=B_{\geq r}(b)$. To find the desired set $U''$ we apply \cite[Proposition 3.8(2)]{HaHaPeGps}. \\
			
			\noindent     \textbf{Case 3: $D_2$ is $\Gamma$ and $\CK$ is $p$-adically closed.}
			
			The proof is similar to the one in Case 2. By \cite[Lemma 3.4]{OnVi}, there exists a $\Gamma$-box  $U_0\subseteq U$ around $b=(b_1,\dots, b_n)$, i.e. a definable set of the form $I_1\times\dots \times I_n$, where for each $i\leq n$, $b_i\in I_i=\{\alpha_i \leq x\leq \beta_i: x\equiv_{N_i} c_i\} $, for some $N_i\in \mathbb N$,  where both intervals $[\alpha_i,b_i]$ and $[b_i,\beta_i]$ are infinite and $0\leq c_i<N_i$.
			
			Let $p(r)$ be the type in $r\in \Gamma$ stipulating that $\dpr(a/Abr)\geq \dpr(a/Ab)$,  that $r>\mathbb{N}$ and that both $b_i<b_i+r<\beta_i$ and $\alpha_i<b_i-r<b_i$ are satisfied. As above the relation $\dpr(a/Abr)=\dpr(a/Ab)$ is type-definable since $D_1$ has algebraic dp-rank. It is easily seen to be consistent. For $r$ satisfying this type,  $b_i\in J_i:=\{ b_i-r < x< b_i+r: x\equiv_{N_i} c_i\}\subseteq I_i$. Now take $U'=\prod_i J_i$.
			
			The existence of the desired set $U''$ goes through an induction on $n$, so we only give the case $n=1$, as the induction step is similar. By compactness we find $\alpha',\beta'\in \Gamma$, $b_1-r <\alpha'<b<\beta'<b_1+r$ with $\alpha', \beta'$ independent generics over all the data. By exchange, $\dpr(b_1/A\alpha', \beta')=1$, and $\alpha'<x<\beta' \land x\equiv_b c_1$ is the set $U''$.      
			
			This completes the proof of Proposition \ref{P: topological pair 2}.
		\end{proof}

		\subsection{Proof of commutation}
		We return to our task of proving that the infinitesimal subgroups commute. Recall that $\widehat \CK\succ \CK$ is a $|\CK|^+$-saturated extension. 

        Recall the following definition from \cite{HaHaPeGps}:
        \begin{definition}
            Two definable sets $X$ and $Y$ are \emph{foreign} if there is no definable infinite $Z\sub X^n\times Y^m$ projecting finite-to-one on both factors. 
        \end{definition}

        A standard argument gives that if $X$ is foreign to $Y$ and to $Z$ then it is foreign to $Y\times Z$.

		\begin{fact}\cite[Proposition 9.3]{HaHaPeGps}\label{F:dist are foreign}
			Any two distinct distinguished sorts $D_1, D_2$ in $\CK$ are foreign. 
		\end{fact}


		
		\begin{lemma}\label{L:trivial intersection}
			Let $G$ be a definable group.	If $D_1$, $D_2$ are distinct distinguished sorts then $\nu_{D_1}(\widehat \CK)\cap \nu_{D_2}(\widehat \CK)=\{e\}$.
		\end{lemma}
		\begin{proof}
			If either $G$ is not locally almost internal to $D_1$ or $D_2$ then the conclusion is straightforward, so we assume that $G$ is locally almost internal to both $D_1$ and $D_2$.
			
			Since any two distinct distinguished sorts are foreign,  the intersection $\nu_{D_1}(\widehat \CK)\cap \nu_{D_2}(\widehat \CK)$ is finite. It follows from Corollary \ref{C: nu divisible} that at least one of the two groups is torsion-free, so the intersection must be trivial.
		\end{proof}

		\begin{definition}
			Two type-definable groups $\rho_1$ and $\rho_2$ \emph{commute} if for any $g\in \rho_1(\widehat \CK)$ and $h\in \rho_2(\widehat \CK)$, $gh=hg$.    
		\end{definition}

		\begin{proposition}\label{P: nus commute}
			Let $G$ be a definable group in $\CK$ and let $D_1$ and $D_2$ be two 
			distinct distinguished sorts. Then $\nu_{D_1}(\widehat \CK)$ and $\nu_{D_2}(\widehat \CK)$ commute.
		\end{proposition}
		\begin{proof}
            We may assume that $G$ is almost strongly internal to both $D_1$ and $D_2$, as otherwise $\nu_D(G)=\{e\}$, and there is nothing more to prove. We may further assume that $G$ is both $D_1$ and $D_2$-balanced. Indeed, if $G$ is not $D_2$-balanced, i.e. $D_2=\KOp$ (Proposition \ref{P:lasi=lsi}) then there is a finite normal  abelian subgroup $H\trianglelefteq G$ with  $G/H$  $D_2$-balanced (Proposition \ref{Canonical H}). Since $K$ is $p$-adically closed,  $D_1$ is necessarily unstable. Let $f:G\to G/H$ be the quotient homomorphism. By definition, $f_*\nu_{D_2}(G)\equiv \nu_{D_2}(G/H)$ and by Proposition \ref{P:nu under finite to one}, $f_*\nu_{D_1}(G)\equiv \nu_{D_1}(G/H)$. 
			
			Assume that $\nu_{D_2}(G/H)$ and $\nu_{D_1}(G/H)$ commute; thus as $f$ is a homomorphism, for any $a\in \nu_{D_2}(G)(\widehat \CK)$ and $b\in \nu_{D_1}(G)(\widehat \CK)$, $aba^{-1}=bh$ for some $h\in H$. Since $H$ commutes with $\nu_{D_1}(G)$ (Proposition \ref{P:general nu of finite index subgroup}), $ab^ka^{-1}=b^k$ for $k=|H|$. By Corollary \ref{C: nu divisible}, $\nu_{D_1}(G)$ is divisible so every element of $\nu_{D_1}(G)$ is of the form $b^k$ for some $k$, and so we are done.
			
			Assume then that $G$ is $D_1$-balanced  and $D_2$-balanced. By Lemma \ref{L:trivial intersection}, it suffices to show  that $\nu_{D_1}(\widehat \CK)$ and $\nu_{D_2}(\widehat \CK)$ normalize each other. By symmetry, it suffices to show that $\nu_{D_1}(\widehat \CK)$ normalizes $\nu_{D_2}(\widehat \CK)$. If $D_2$ is stable then $\nu_{D_2}(\widehat \CK)$ is a definable normal subgroup of $G(\widehat \CK)$, so obviously $\nu_{D_1}(\widehat \CK)$ normalizes $\nu_{D_2}(\widehat \CK)$. We assume that $D_2$ is unstable and we prove that $\nu_{D_1}(\widehat \CK)$ normalizes $\nu_{D_2}(\widehat \CK)$.
			
			We will actually show something stronger: For every definable $U$ with  $\nu_{D_2}\vdash U$, there exists a $D_1$-critical subset $W\subseteq G$ and some $V$, $\nu_{D_2}\vdash V$,  such that for every $\beta \in WW^{-1}$, $V^\beta\subseteq U$. Since $\nu_{D_1}\vdash WW^{-1}$  (Proposition \ref{P: nu withough generics}), it will follow that for every $\beta \in \nu_{D_1}(\widehat \CK)$ there exists $V$, with $\nu_{D_2}\vdash V$, such that $V^\beta\subseteq U$. Therefore, $\nu_{D_2}(\widehat \CK)^\beta \subseteq \nu_{D_2}(\widehat \CK)$, as desired.

			Fix some  $D_2$-critical subset $Y_0\subseteq G$ and a $D_1$-critical subset $X_0\subseteq G$ and let $(a,b)\in X_0\times Y_0$ be a generic pair over a parameter set $A$ over which everything is defined.  By Corollary \ref{C: has D-sets}, we may identify $Y_0$ with a subset of $D_2^n$, $n=\dpr(Y_0)$, and identify $X_0$ with a subset of $D_1^m$ for $m=\dpr(X_0)$.
		
			By Proposition \ref{P: nu withough generics}:     $\nu_{D_2}\equiv \{YY^{-1}: Y\in \nu_{Y_0}(b)\}$. 
			Fix some $U=YY^{-1}$ for some $Y\in \nu_{Y_0}(b)$.
			Notice that $\nu_{Y_0}(b)\equiv \nu_Y(b)\equiv \nu_{D_2^n}(b)$ (Fact \ref{F: vicinities}(2)).
			
			By Fact \ref{F: properties of nu}, every element of $G=G(\CK)$ normalizes $\nu_{D_2}(\widehat \CK)$, hence there exists  $Y_1\in \nu_{Y_0}(b)$ such that $(Y_1Y_1^{-1})^{a^{-1}}\sub U$. By Proposition \ref{P: topological pair 2}, we may shrink $Y_1$ if needed, and  assume that $\dpr(a/Ab[Y_1])=m$.

			Next, using again the fact that $a$ normalizes $\nu_{D_2}(\widehat \CK)$, there is $Y_2\in \nu_{Y_0}(b)$, such that $(Y_2Y_2^{-1})^a\sub Y_1Y_1^{-1}$. Again, using Proposition \ref{P: topological pair 2}, we may assume, after shrinking $Y_2$ that $\dpr(a/Ab[Y_1][Y_2])=m$.
			
			Let \[W=\{x\in X_0: (Y_1Y_1^{-1})^{x^{-1}}\sub U \wedge (Y_2Y_2^{-1})^x\sub Y_1Y_1^{-1}\}.\] The set $W$ contains $a$, is contained in $X_0$, and is defined over $A, [Y_1],[Y_2]$, therefore $\dpr(W)=m$, so $W$ is $D_1$-critical. We claim that for every $x\in WW^{-1}$, we have  $(Y_2Y_2^{-1})^x\sub U$. Indeed, given any $x=x_1x_2^{-1}\in WW^{-1}$, we have
			\[(Y_2Y_2^{-1})^{x_1x_2^{-1}}=((Y_2Y_2^{-1})^{x_1})^{x_2^{-1}}\sub (Y_1Y_1^{-1})^{x_2^{-1}}\sub U.\] Thus, taking $V=Y_2Y_2^{-1}$, we found a $D_1$-critical set $W$ such that for every $\beta\in WW^{-1}$, $V^\beta\sub U$. This completes the proof that $\nu_{D_1}$ normalizes $\nu_{D_2}$ and with it the proposition.
		\end{proof}

		\begin{example}
			Let $\CK$ be any of our valued fields. Consider the semi-direct product $G:=K/\CO\rtimes \CO^\times$; the multiplication law is given by 
			\[(a,b)*(c,d)=(a+bc,bd).\]
			It is locally strongly internal to $K/\CO$ and to $K$. Since the distinct distinguished sorts are foreign, it is not hard to verify that $G$ is $K/\CO$-balanced (by Proposition \ref{P:lasi=lsi} it is automatically $K$-balanced). We get that $\nu_{K/\CO}(G)=\nu_{K/\CO}(K/\CO)\times \{1\}$ and $\nu_K(G)=\{0\}\times \nu_K(\CO^\times)$, where $\nu_{K/\CO}(K/\CO)$ is the infinitesimal subgroup of $K/\CO$ and   $\nu_K(\CO^\times)$ is the multiplicative infinitesimal subgroup of $\CO^\times$ (around $1$).
			
			It is not hard to see that any $(a,b)\in G(\CK)$ normalizes the infinitesimal groups, we leave the calculations to the reader. The aim here is to verify that the infinitesimal groups  commute.
			
			The following computations are straightforward: For any $(a,b)\in G$, $(a,b)^{-1}=(b^{-1}\cdot (-a),b^{-1})$ and 	$(0,b)(c,1)(0,b)^{-1}=(bc,1).$
			
			Let $c\in \nu_{K/\CO}(K/\CO)(\widehat \CK)$ and $b\in \nu_K(\CO^\times)(\widehat \CK)$. Note that $v(bc-c)=v(c(1-b))\in \CO(\widehat \CK)$ so $bc=c$ (in $K/\CO$).
		\end{example}

		\section{The main theorem and rank (in)equality}
		
		We are now ready to prove our main theorem. 
		
		\begin{theorem}\label{T:main}
			Let $\CK$ be an expansion of a valued field of characteristic $0$ that is either (i) $V$-minimal, (ii) power bounded $T$-convex or (iii) $p$-adically closed and let $G$ be an infinite $\CK$-definable  group. Let  $\widehat \CK\succ \CK$ be a $|K|^+$-saturated elementary extension.    
			
			The group $\nu(\widehat \CK)$ generated by $\nu_K(\widehat \CK),\, \nu_\Gamma (\widehat \CK),\, \nu_{K/\CO}(\widehat \CK),\, \nu_\bk(\widehat \CK)$ in $G(\widehat \CK)$ is an internal direct product, and thus type definable.
			
		\end{theorem}
		\begin{proof}
			By  Proposition \ref{P: nus commute}, any two of the $\nu_D$ commute. Let $N$ be the subgroup generated by the $\nu_D(\widehat \CK)$ in $\widehat \CK$.
			We show that the map $\tau: \nu_K(\widehat \CK)\times \nu_\Gamma (\widehat \CK) \times \nu_{K/\CO}(\widehat \CK)\times \nu_\bk(\widehat \CK)\to N$ given by $(a,b,c,d)\mapsto abcd$ is a group isomorphism. 
			
			The map $\tau$ is a group homomorphism (since the $\nu_D$ commute) and it is obviously onto. To show that it is injective, we will need to prove a generalization of Lemma \ref{L:trivial intersection}. For ease of presentation we assume that all the $\nu_D$ are non-trivial, the general case follows similarly. It will be convenient to deal with $K/\CO$ last. 
			
			By Lemma \ref{L:trivial intersection}, $\nu_\Gamma(\widehat \CK)$ has trivial intersection with $\nu_{\bk}(\widehat \CK)$. Consider $S:=\nu_K(\widehat \CK)\cap \left(\nu_\Gamma(\widehat \CK)\cdot \nu_\bk(\widehat \CK)\right)$. By what we have just shown we can identify $S$ with a subgroup of $\nu_\Gamma(\widehat \CK)\times \nu_\bk(\widehat \CK)$. Since both $K$ and $\Gamma$ and $K$ and $\bk$ are foreign it is not hard to see that $K$ and $\Gamma\times \bk$ are foreign. It follows that $S$ is finite. So $S$ is a finite subgroup of $\nu_K$. %
			But $\nu_K$ is torsion-free by Corollary \ref{C: nu divisible}(2), so must be trivial.

			Finally, consider $\nu_{K/\CO}(\widehat \CK)\cap (\nu_K(\widehat \CK)\cdot \nu_\Gamma(\widehat \CK)\cdot \nu_\bk(\widehat \CK))$. A similar argument to the one given in the previous paragraph shows that the intersection must be finite. If $\CK$ is not $p$-adically closed then $\nu_{K/\CO}$ is torsion-free and if $\CK$ is $p$-adically closed then the group on the right is torsion-free (again by Corollary \ref{C: nu divisible}(2)), either way the intersection is trivial.
			
			Consequently, $\ker(\tau)$ is trivial and $\tau$ is an isomorphism.
		\end{proof}

    In view of the above result we introduce: 
        \begin{definition}
            Let $G$ be an infinite $\CK$-definable group. The \textit{infinitesimal subgroup of $G$}, denoted $\nu(G)$, is the subgroup generated by the $\nu_D$, as $D$ ranges over all four distinguished sorts.  
        \end{definition}
		
		\subsection{The rank equality and EI}
	Our main theorem can be restated as showing that $\nu(G)$ is type-definable and definably isomorphic to the direct product of the $\nu_D$. The following is an immediate consequence: 
		
		\begin{proposition}\label{C:Cor of main}
			Let $\CK$ be an expansion of a valued field of characteristic $0$ that is either (i) $V$-minimal, (ii) power bounded $T$-convex or (iii) $p$-adically closed and let $G$ be an infinite definable group. For $\CD=\{K,\Gamma,K/\CO,\bk\}$, we have  $1\le \dpr(\nu(G))=\sum_{D\in \CD} \adrk(G)\leq \dpr(G).$
		\end{proposition}

        In light of the above proposition, it  is natural to ask to what extent the group $\nu(G)$ ``exhausts'' the group $G$, namely whether $\dpr(\nu(G))=\dpr(G)$ (below we call it ``the rank equality''). There are a few examples where we know the rank equality to hold: When $G$ is dp-minimal this follows from the mere fact that $\nu(G)$ is infinite. When $G$ is definably semisimple (i.e., has no infinite definable commutative normal  subgroups) the rank equality  follows from   \cite[Theorem 10.3(1,2)]{HaHaPeSemisimple}.  
		
		Below we show that the  rank equality fails in general, even for $K$-pure groups (i.e., where $\nu(G)=\nu_K(G)$). Classifying those cases where the rank equality holds is, in view of the above examples, an intriguing problem in its own right.  It  seems to have a relation, though yet to be fully understood,  with elimination of imaginaries in $\CK$. 
		
		For a more detailed discussion we need Gagelman's extension (\cite{Gagelman}) of the $\acl_K$-dimension to $\CK^{eq}$. See also \cite[\S 2] {JohnTopQp} for a quick overview of the definition, and main properties. In those terms  \cite[\S 4.1]{CasHasTconv} introduces the notion of a valued field $\CK$  eliminating  imaginaries down to 0-dimensional sorts. This means that for any 0-definable equivalence relation $E$ on $K^n$ (any $n$) there is a $0$-definable injection $f: K^n/E\to K^m\times S$ for some $m\in \Nn$ and some $0$-dimensional sort $S$.  The main result of \cite{HaHrMac1} asserts that the theory of pure algebraically closed fields eliminates imaginaries down to the so-called, geometric sorts, that are readily checked to be $0$-dimensional. Similar results exist for the theories of real closed valued fields (\cite{MelRCVFEOI}) and $p$-adically closed fields (\cite{HrMarRid}) in the language of valued rings. 
		
		In the setting of the present paper, if $\CK$ has elimination of imaginaries down to $0$-dimensional sorts, and $G$ is a pure $K$-group  then the rank equality holds. Indeed, by elimination of imaginaries down to 0-dimensional sorts, for such a group $G$  there exists a definable finite-to-one map from $G$ into some $K^n$   (\cite[Lemma 4.1]{CasHasTconv}); so, for $D=K$, we have $\dpr(G)=\adrk(G)$  and since $\adrk(G)=\drk(G)$ it follows that $\dpr(\nu(G))=\dpr(\nu_K(G))=\dpr(G)$. 
		In general, however, the rank equality need not hold even for a pure $K$-group, as the following example shows: 
		
		\begin{example}\label{E: exp example}
			
			Let $\CK$ be a model of an analytic expansion of ACVF$_{0,0}$ (in the sense of \cite{Lip}, see also \cite[\S 1]{HaHrMac3} for a quick overview). Let $\exp: \m \to 1+\m$ be the exponential map (definable in this structure) and let $E$ be its graph. If $\widehat G=(\CO,+)\times (\CO^\times,\cdot)$ then $E$ is a subgroup. We define $G=\widehat G/E$, and let  $\pi:\widehat G\to G$ be the quotient map.
			Restricting the residue map $\mathrm{res}$ to $\CO^\times$, it is a group homomorphism from $\CO^\times$ onto $\bk^\times$ with kernel $1+\m$.  Since $E\sub \m\times (1+\m)$ , the map $\mathrm{res}:\widehat G \to \bk \times \bk^\times$ factors through $\pi$, and in particular, $\dpr(G)=\dpr(\bk\times \bk^\times)=2$. Let $f:G\to \bk\times \bk^\times$ be the corresponding homomorphism.

			Our goal is to prove that the $K$-rank of $G$ is $1$, but $G$ contains no infinite definable set of dimension $0$, so in particular, for all distinguished sorts $D\neq K$, $\drk(G)=0$.

			Because $\dim(E)=1$, it follows immediately from the definition of $\dim$ that for a generic $(a,b)\in \widehat G$, we have $\dim(\pi(a,b)/\emptyset)=1$ and thus for every definable  $Y\sub \widehat G$, if $\dim(Y)=2$ then $\dim(\pi(Y))=1$. In particular, $\dim(G)=1$, and we can see directly that the $K$-rank of $G$ is $1$, because the restriction of $\pi$ to either $\{0\}\times \CO^\times$ or $\CO\times \{1\}$ is injective.

	Let us now show: 
			\begin{claim}\label{C: no 0-dim sets}
				There is no definable infinite $X\sub G$ with $\dim(X)=0$.     
			\end{claim}
			\begin{claimproof}
				Assume otherwise, and let $X$ witness it.
				By the above discussion, $\dpr(X)=1$ (indeed, otherwise its preimage in $\widehat G$ has dimension $2$, implying that $\dim(X)=1$). 
				
				We claim that for $\widehat X=\pi^{-1}(X)$, the projections $\pi_i(\widehat X)$, $i=1,2$, into $\CO$ and $\CO^\times$, respectively, contain infinitely many cosets of $\m$ and $1+\m$, respectively. Indeed, $\widehat X$ consists of infinitely many cosets of $E$, and since $\exp:\m \to 1+\m$ is surjective, if $\pi_2(\hat X)$ intersected only  finitely many cosets of $1+\m$ then for some generic $b\in \pi_2(\widehat X)$ the fiber $\widehat X^b\sub \CO$ would be infinite. It thus would follow from additivity of $\dim$ that $\dim(\widehat X)=2$. Similarly, if $\pi_1(\widehat X)$ consisted of finitely many cosets of $\m$ then $\dim(\widehat X)=2$.  In either case we would obtain a contradiction to the fact, follows by additivity, that $\dim(\widehat X)=1$. Note that the above proof uses only the fact that $X$ is infinite and $0$-dimensional.   

                We now claim that $f(X)$ is infinite. Indeed, since both $\pi_1(\widehat X)$ and $\pi_2(\widehat X)$ intersect infinitely many cosets of $\m$ and $1+\m$ respectively, it follows that $\widehat X$ intersects infinitely many cosets of $\m\times (1+\m)$.  Therefore, $f(X)$ is infinite. Now, since $\bk$ is a pure algebraically closed field, we may assume (after possibly shrinking $X$ and translating it in $G$) that $f(X)$ is an an infinite irreducible $\bk$-constructible subset of the group $\bk\times \bk^\times$, containing the identity element $(0,1)$. 

                By standard facts on algebraic groups, see \cite[Proposition 2.2.6]{AlgGroups},  the subgroup $L$ of $\bk\times \bk^\times$ generated by $f(X)$ is connected and algebraic, and there exists $n\in \mathbb N$ such that $L=f(X^{\langle n\rangle })$ (where $X^{\langle n\rangle }$ is the set of group products of $n$ elements from $X$ and $X^{-1}$).
				A non-trivial connected algebraic subgroup of $\bk\times \bk^\times$ is either  $\bk\times \bk^\times$, $\{0\}\times \bk^\times$ or $\bk\times \{1\}$. Since $\dim(X)=0$, then   $\dim (X^n)=0$, so also $\dim(X^{\langle n\rangle})=0$. Consequently, as in the first paragraph of the proof of the claim, it follows that  $\dpr(X^{\langle n\rangle })=1$, so only the last two are possible. Thus, either $f(X^{\langle n\rangle )})=\{0\}\times \bk^*$ or $f(X^{\langle n\rangle })=\bk\times \{1\}$.  Since $f$ factors through $\pi$, in the first case we have 
            $\pi^{-1}(X^{(n)})=f^{-1}(\{0\}\times \bk^*)=\bm\times \CO^\times$ and in the second case $\pi^{-1}(X^{(n)})=f^{-1}(\bk\times \{1\})=\CO\times (1+\bm)$.

				Since $\dim(X^{\langle n\rangle})=0$ it follows, as we proved for $\pi^{-1} (X)$ above, that the projections of $\pi^{-1}(X^{\langle n\rangle })$ onto both coordinates must contain infinitely many cosets of $\bm$ and $1+\bm$, respectively. This contradicts the above two cases.
                
				
                
				

				
				Therefore, there are no infinite definable subsets of $G$ of dimension $0$.
			\end{claimproof}

			It follows from the claim that for all $D\neq K$,  $\drk(G)=0$. Thus,  $G$ is $K$-pure and $\nu(G)=\nu_K(G)$, so $\dpr(\nu(G))=1<\dpr(G)=2$. 
		\end{example}
		
		In \cite[Theorem 1.1]{HaHrMac3} it is shown, using an example similar to the above, that the field $\CK$ of the previous example does not have elimination of imaginaries down to the geometric sorts of \cite{HaHrMac1}. In fact, as noted in \cite[\S 4.1]{CasHasTconv}, the example of Haskell, Hrushovski and Macpherson shows that $\CK$ does not have elimination of imaginaries down to 0-dimensional sorts. Below we give a quick alternative proof of this result, based on the machinery studied in the present paper. As we will see, the proof is an almost immediate corollary of the above example:  
		\begin{corollary}
			Let $\CK$ be a model of an analytic expansion of ACVF$_{0,0}$. Then $\CK$ does not have elimination of imaginaries down to 0-dimensional sorts. 
		\end{corollary}
		\begin{proof}
			Assume towards a contradiction that $\CK$  eliminates  imaginaries down to $0$-dimensional sorts. This means that for the group $G$ of Example \ref{E: exp example}, there exists a definable injection $f: G\to K^n\times S$ for some $0$-dimensional sort $S$. Identifying $G$ with $f(G)$, we may assume that $G\subseteq K^n\times S $.
			
			Let $\pi_1: G\to K^n$ be the projection. If for some $a\in K^n$ the fiber $\pi_1^{-1}(a)$ is infinite then the set $\{g\in G: \pi_1(g)=a\}$ is infinite and $0$-dimensional (as it is in definable bijection with a subset of the $0$-dimensional set $S$), contradicting Claim \ref{C: no 0-dim sets}.  
			
			Therefore, the function $\pi_1: G\to K^n$ is finite-to-one. But this too is impossible, as we have seen that $\dpr(G)=2$, and thus also $\dpr(\pi_1(G))=2$. As $\pi_1(G)\sub K^n$ it follows that \[2=\dpr(\pi_1(G))=\dim(\pi_1(G))\le \dim(G)=1, \]
			a contradiction. 
		\end{proof}

		Let us note that in \cite[Theorem 1.2, Theorem 1.3 ]{HaHrMac3}  results analogous to the above corollary are proved for the $T$-convex expansion of $\Rr_{an}$ and for $\mathbb Q_{p,an}$ respectively. However, since $\exp\restriction [-1,1]$ is definable in $\Rr_{an}$, the analogous example in the $T_{an}$-convex setting yields a group $G$ which contains a definable infinite subset of dimension $0$ (the projection of the graph of  $\exp\restriction [-1,1]$ into $G$) which is strongly internal to $\bk$. Therefore,  in that case $\dpr(\nu_K(G))=\dpr(\nu_k(G))=1$ and as $\nu(G)\cong\nu_K(G)\times \nu_{\bk}(G)$  the rank equality holds.\\

		The following problems remain open: 
		
		\begin{question} 
			\begin{enumerate}
				\item Does the rank equality hold in ACVF$_{0,0}$,
				in T-convex expansions of power bounded o-minimal fields and in p-adically closed fields?
				\item Does an example similar to Example \ref{E: exp example} exist in $\Qq_{p,an}$? 
			\end{enumerate}
			
		\end{question}

		\section{Infinitesimals of subgroups}
		%
		
		The aim of the present section is to show that if $G_1\le G$ is a definable subgroup, then $\nu(G_1)$ is a relatively definable subgroup of $\nu(G)$, namely the intersection of a definable set with $\nu(G)$. 
		We first prove this result for each distinguished sort separately. 
		
		\subsection{Infinitesimals of lower dp-rank}
		We consider an unstable distinguished sort $D$ and a generic $a\in D^n$. Given an $A$-definable set $Y\sub D^n$,  possibly of lower dp-rank, but such that $a$ is $A$-generic in $Y$ we first investigate the relation between $\nu_{D^n}(a)$ and $\nu_Y(a)$. Recall that if $\dpr(Y)=n$ then by Fact \ref{F: vicinities}(2) we get $\nu_{D^n}(a)=\nu_Y(a)$. 

			Below, for a partial type $\mu$, and a definable set $X$, we write, in abuse of notation, $\mu \cap Y$ for the collection $\{X\cap Y:X\in \mu \}$.
			
			\begin{lemma}\label{L: restriction of nu to subset}
				Let $D$ be an unstable distinguished sort, $Y\subseteq D^n$ an infinite definable subset,  $c\in Y$, with $\dpr(c/[Y])=\dpr(Y)$ and $\dpr(c/\emptyset)=n$.
				
				\begin{enumerate}
					\item If $D\neq \KOp$ then, $\nu_Y(c)\equiv \nu_{D^n}(c)\cap Y$.
					\item If $D=\KOp$ then there exists a generic vicinity $Y'$ of $c$ in $Y$, which is a coset of a definable subgroup of $D^n$, such that $\nu_Y(c)\equiv \nu_{D^n}(c)\cap Y'$.
					
				\end{enumerate}
			\end{lemma}
			\begin{proof}
				We break the proof into cases.
				
				\vspace{.2cm}
				\noindent   \textbf{Case 1:  $D$ is an SW-uniformity.} 
				\vspace{.2cm}

				By Fact \ref{F: nu in SW for general Y} (see also \cite[Proposition 5.9]{HaHaPeGps}), $\nu_{D^n}(c)$ is logically equivalent to \[\{U\subseteq D^n: \text{$U\ni c$ is definable open}\}\] and $\nu_Y(c)$ is logically equivalent to $\{V\subseteq Y: \text{$V\ni c$ is relatively definable open in $Y$}\}$.
				Since the topology on $Y$ is the subspace topology, the result follows.

				\vspace{.2cm}
				
				\noindent    \textbf{Case 2: $D=\Gamma$ and $\CK$ is $p$-adically closed.} 
				
				\vspace{.2cm}

				By \cite[Lemma 5.5]{HaHaPeSemisimple}, $\nu_{D^n}(c)\equiv \{U : U\subseteq D^n \text{ is a $\Gamma$-box around $c$}\}$. 
				It is not hard to see that $\nu_Y(c)\vdash \nu_{D^n}(c)\cap Y$. Indeed, if $U\sub D^n$ is a $\Gamma$-box around $c$ then, by compactness we can find a  smaller $\Gamma$-box $U'$ around $c$,  and whose end points are independent of all the data, so by exchange, $c$ is generic in $U'\cap Y$ and $U'\cap Y\in \nu_Y(c)$ (see also \cite[proof of Lemma 4.2]{HaHaPeGps}). For the converse, it is enough to prove the following:
				
				\begin{claim} 
					For every $Z\in \nu_Y(c)$ there  is a $\Gamma$-box $U\sub \Gamma^n$ around $c$ such that $U\cap Y\sub Z$. 
				\end{claim} 
				\begin{claimproof} 
					We use induction on $n$. Let $m=\dpr(Y)$.
					Since $\nu_Y(c)$ is a filter-base (Fact \ref{F: vicinities}),  by shrinking $Z$ if needed, we may assume that $\dpr(c/A[Z][Y])=\dpr(Y)$. 
					
					If $m=n$ (this covers the $n=1$ case) then by \cite[Lemma 3.2]{OnVi} we further shrink $Z$ and assume that it is a $\Gamma$-box so $U=Z$ works. We thus assume that $1\leq m<n$.
					
					Since dp-rank is algebraic in $\Gamma$, $c$ has a sub-tuple $c'$ of length $m$ and of dp-rank $m$ over $A[Z][Y]$ 
					such that $c\in \dcl(c'A[Z][Y])$. Without loss of generality, $c'=(c_1,\ldots, c_m)$. Let $\pi:\Gamma^n\to \Gamma^m$ be the projection on the first $m$ coordinates. Since $c'\in \pi(Y)$ it follows from additivity of rank in $\Gamma$ that $\pi^{-1}(c')\cap Y$ is finite.

					For simplicity of presentation assume that $n=m+1$, the general case is similar.  Let the finite fiber   $\pi^{-1}(c')\cap Y$ be $(c',d_1),\dots (c',d_k)$, with  $d_1=c_n$.

					Let $N\in \mathbb N$ be such that for all $i=2,\ldots, k$, whenever $d_i-d_1\in \mathbb{Z}$, we have $d_i-d_1\notin N \Gamma$, and let $\alpha<d_1<\beta\in \Gamma$ be elements with $\beta-\alpha\notin \mathbb{Z}$ such that whenever $d_i-d_1\notin \mathbb Z$, then $d_i$ is not between $\alpha$ and $\beta$. We may choose such $\alpha,\beta$ so that $\dpr(c/A[Z][Y]\alpha,\beta)=m$.  
					Let $B=A[Z][Y]\alpha,\beta$.
					
					Fix $r\in \mathbb N$  such that $d_1-r\in N\Gamma$. Define for $x\in \Gamma^m$, 
					\[  R_{x}:= \{y\in \Gamma: (x,y)\in Y \wedge  \alpha<y<\beta\wedge  y-r\in N\Gamma\},\]
					and note that   $R_{c'}=\{d_1\}.$      Let   
					\[R=\{(x,y)\in \Gamma^m\times \Gamma: y\in R_{x}\wedge \exists\, ! \,w \in R_{x}\}.\]
					
					The set $R$ contains $c$ and is defined over $B$, thus $\dpr(R)=m$ and as $R\cap Z$ contains $c$, $\dpr(R\cap Z)=m$ as well. Since $R$ projects injectively into $\Gamma^m$, $c'\in \pi(R\cap Z)\sub \pi(Y)$, and $c'$ is $B$-generic in both sets. By induction, there is a $\Gamma$-box $B_1\sub \Gamma^m$ around $c'$, such that $B_1\cap \pi(Y)\sub \pi(R\cap Z)$.
					
					Let $U$ be the $\Gamma$-box $B_1\times \{y\in \Gamma: \alpha<y<\beta\wedge  y-r\in N\Gamma\}$. We claim that $U\cap Y\sub Z$. Indeed, if $(x,y)\in U\cap Y$ then $x\in B_1\cap \pi(Y)$, so $x\in \pi(R\cap Z)$. It follows, by definition of $R$, that there is a unique $w\in R_{x}$ such that $(x,w)\in R\cap Z$. Also, by the definition of $U$, we have $y\in R_{x}$, so necessarily $y=w$, namely $(x,y)\in Z$, as claimed.
				\end{claimproof}
				
				This ends the proof in Case 2.

				\vspace{.2cm}

				\noindent\textbf{Case 3:} $D=\KOp$.
				\vspace{.2cm}

				By Fact \ref{F: concrete description of nu for D-balanced}(2) and Remark \ref{R: how nu looks}(2), $\nu_{D^n}(c)\equiv \{U\subseteq D^n: U \text{ is a  ball around $c$}\}$. 
				By Lemma \ref{L:onto open} (and Remark \ref{R: onto open for K/O padic}), there exists a generic vicinity $Y'\subseteq Y$ of $c$ in $Y$, which is a coset of a definable subgroup of $D^n$, and a coordinate projection $\pi:D^n\to D^m$ such that $\pi\restriction Y'$ is injective onto a ball $B\subseteq D^m$ around $\pi(c)$. By Fact \ref{F: vicinities}(2), $\nu_Y(c)\equiv \nu_{Y'}(c)$. By Fact \ref{F: concrete description of nu for D-balanced}(2), and using the fact that $\pi\rest Y'$ is injective, $\nu_{Y'}(c)\equiv\{\pi^{-1}(V)\cap Y': \text{$V\subseteq B\subseteq D^m$ is a ball around $\pi(c)$}\}$. The result follows.
			\end{proof}
			
			The above lemma is optimal, in the sense that, when $D= \KOp$,  passing to  a subset $Y'$ of $Y$ in (2) of the above lemma may be unavoidable:
			
			\begin{example}
				Assume that $\CK$ is $p$-adically closed. Let $(a,b)\in K/\CO\times K/\CO$ be a generic pair and $c\in K/\CO$ be such that $v(a-c)=-1$; note that $a$ and $c$ are inter-algebraic. Letting $X=K/\CO\times K/\CO$ we get $\nu_X(a,b)\equiv (\nu_{K/\CO}(K/\CO,+))^2$. For $Y=\{a,c\}\times K/\CO$ we thus have that $\nu_X(a,b)\cap Y$ is logically equivalent to $\{a,c\}\times \nu_{K/\CO}(K/\CO,+)$. But since $\{a\}\times K/\CO$ is a generic vicinity of $(a,b)$ in $Y$, we get that $\nu_Y(a,b)\equiv \{a\}\times \nu_{K/\CO}(K/\CO,+)$.
			\end{example}
			
			\subsection{Infinitesimals of subgroups}
			We now turn our attention to the infinitesimals of subgroups. We start with a useful technical observation. 
			\begin{lemma}\label{L:subgroup is also balanced}
				Let $G$ be a definable group in $\CK$ and  $G_1\leq G$ a subgroup with $\adrk(G_1)>0$. If $G$ is $D$-balanced then so is $G_1$.
				
			\end{lemma}
			\begin{proof}
				
				Assume that $G$ is $D$-balanced, so every $D$-critical subset of $G$ is almost $D$-critical. Let $Y$ be a $D$-critical subset of $G$ and $Z$ an almost $D$-critical subset of $G_1$. By Lemma \ref{L:locally critical}, there is $g\in G$ such that $gY$ contains $X:=Z\cap gY$ and the latter is an almost $D$-critical set of $G_1$. But then $\dpr(X)=\adrk(G_1)$ and $X$ is contained in the domain of an injective function into some $D^n$ (namely,  the function witnessing that $gY$ is strongly internal to $D$), so $X$ is $D$-critical.
			\end{proof}
			
			We also need the following observation. 
			\begin{lemma}\label{L: same H if same adrk}
				Let $G$ be a $\CK$-definable group, $G_1\le G$ a definable subgroup, and let $D$ be a distinguished sort. Then $H_{G_1}\le H_G$, and if $\adrk(G_1)=\adrk(G)$ then $H_G=H_{G_1}$. 
			\end{lemma}
			\begin{proof}
				By Proposition  \ref{P:lasi=lsi}, we may assume that $D=\KOp$.
				
				If $\adrk(G)=\adrk(G_1)$  then, since every almost $D$-critical set  of $G_1$ is also an almost $D$-critical of $G$, we conclude by  \cite[Proposition 4.35(2)]{HaHaPeGps} that $H_G=H_{G_1}$.

				For the general case,  since $H_G$ is normal in $G$, it follows that $G_2:=G_1H_G$ is a subgroup of $G$ and as  $G/H_G$ is $D$-balanced, it follows from Lemma \ref{L:subgroup is also balanced} that  so is $G_2/H_G$. By minimality of $H_{G_2}$, $H_{G_2}\leq H_G$. As $\adrk(G_2)=\adrk(G_1)$, by the previous paragraph, $H_{G_2}=H_{G_1}$, so $H_{G_1}\sub H_G$.
			\end{proof}
			
			The following example shows $H_{G_1}$ may be a proper subgroup of $H_G\cap G_1$.
			\begin{example}
				Let $G_0:=(K/\CO)/C_p$ be as in Example \ref{E:not balanced}, and consider  $G=G_0\times G_0$ with the subgroup  $G_1=G_0\times H_{G_0}$. Note that, since $G_0\times \{0\}$ has finite index in $G_1$ we get that $H_{G_1}=H_{G_0}\times \{0\}$. Also, as $G/(H_{G_0}\times H_{G_0})$ is $K/\CO$-balanced, $H_G\leq H_{G_0}\times H_{G_0}$ so $H_G\leq G_1$. 
				
				But $G/H_{G_1}\cong K/\CO \times G_0$, which is not $K/\CO$-balanced. So $H_{G_1}\lneq H_G=H_G\cap G_1$. 
				
			\end{example}

			We now prove the relative definability results for $\nu_D(G_1)$ in a series of lemmas. Recall that when $D$ is stable then $\nu_D(G)$ is a definable connected normal subgroup of $G$.
			
			\begin{lemma}\label{L: nu of subgroup when stable}
				Let $G$ be a definable group in $\CK$ $V$-minimal and  $G_1\leq G$ a definable subgroup. If $D=\bk$  then $\nu_D(G_1)$ has finite index in $\nu_D(G)\cap G_1$. 
			\end{lemma}		
			\begin{proof}	
            If $\drk(G_1)=0$ then $\nu_D(G_1)=\{e\}$ and $\nu_D(G)\cap G_1$  is finite so we are done. So we assume that $\drk(G_1)>0$, and by Proposition \ref{P:lasi=lsi}, $G$ and  $G_1$ are $D$-balanced. 
				
				Since $D=\bk$ in the $V$-minimal setting, the groups $\nu(G)$, $\nu(G_1)$ are definably isomorphic to connected $\bk$-algebraic groups, normal in $G$ and $G_1$, respectively. By Elimination of Imaginaries in $\bk$, the group $N=\nu(G)\cdot \nu(G_1)$  is also definably isomorphic to a connected $\bk$-algebraic group, hence $\dpr(N)\leq \drk(G)=\dpr(\nu(G))$. Since dp-rank agrees with $\mr$ in $\bk$, it follows that $N=\nu_D(G)$. Hence $\nu_D(G_1)\leq  \nu_D(G)$ and so $\nu_D(G_1)\sub \nu_D(G)\cap G_1$. 
				
				However, the group $\nu_D(G)\cap G_1$ is also definably isomorphic to a $\bk$-algebraic group (as a definable subgroup of $\nu_D(G)$), so its $\dpr$ is at most $\drk(G_1)$. Since $\mr(\nu_D(G_1))=\dpr(\nu_D(G_1))=\mr(\nu_D(G)\cap G_1)$ it follows that $\nu_D(G_1)$ has finite index in $\nu_D(G)\cap G_1$.
			\end{proof}
			
			Obviously, the previous lemma cannot be improved. Working inside the algebraically closed field $\bk$, $G_1$ can be a non-connected algebraic group so the finite index appears naturally.

			As is often the case, we separate  $\KOp$ from the remaining cases so we first prove the following.
			
			\begin{lemma}
				Let $G$ be a definable group in $\CK$,  $G_1\leq G$ a definable subgroup, and $D$ is unstable different than  $\KOp$. Then  $\nu_D(G_1)\equiv \nu_D(G)\cap G_1$.
			\end{lemma}
			
			\begin{proof}
				Assume first that $\drk(G_1)=0$. Then $\nu_D(G)\cap G_1 $ must be finite, but by Lemma \ref{L:nu divisible}, $\nu_D(G)$ is torsion-free hence $\nu_D(G_1)=\{e\}\equiv \nu_D(G)\cap G_1$. So assume $\drk(G_1)>0.$

				Let $X\sub G$ be $D$-critical.
				
				\begin{claim} 
					There exists a parameter set $A$, containing $[X]$, $g\in X$ $A$-generic and  $h\in G_1$ with $\dpr(g,h/A)=\drk(G)+\drk(G_1)$ such that, for $a=h^{-1}g$ the following hold:
					
					(i) $G_1\cap Xa^{-1}$ is $D$-critical in $G_1$ and $h$ is generic in it over $Aa$.
					
					(ii) $g$ is generic in $G_1a\cap X$ over $Aa$.
					
				\end{claim}
				
				\begin{claimproof}
					
					Let $Y\subseteq G_1$  be $D$-critical for $G_1$. Let $A$ be a parameter set over which everything is defined (say containing $[X]$) and  let $(h,g)\in Y\times X$ be generic over $A$.

					For  $a=h^{-1}g$, we have $h\in Y\cap Xa^{-1}$. By 
					Lemma \ref{L:locally critical}, $h$ is generic in that set over $Aa\,$ and in addition $\dpr(Y\cap Xa^{-1})=\dpr(Y)$. 
					Consider now the set $G_1\cap Xa^{-1}$. It is strongly internal to $D$ (since $Xa^{-1}$ is) and its rank is at least $\dpr(Y\cap Xa^{-1})$, which equals $\drk(G_1)$. Hence $\drk(G_1\cap Xa^{-1})=\drk(G_1)$, so $G_1\cap Xa^{-1}$ is $D$-critical for $G_1$.

					Since $h$ and $g$ are inter-definable over $a$ it follows that $$\dpr(g/Aa)=\dpr(h/Aa)=\dpr(G_1\cap Xa^{-1}).$$ Clearly, $\dpr(G_1a\cap X)=\dpr(G_1\cap Xa^{-1})$ and as  $g\in G_1a\cap X$ it follows that  $g$ is $Aa$-generic there. 
				\end{claimproof}
				
				By definition,  \begin{equation}\label{eq:subgroup}\nu_D(G_1)\equiv \left(\nu_{G_1\cap Xa^{-1}}(h)\right) h^{-1}\equiv \left(\nu_{G_1\cap Xa^{-1}}(h)\right) ag^{-1}.\end{equation} 
				The map $x\mapsto xa$ is defined over $aA$, and sends $h$, which is generic in $G_1\cap Xa^{-1}$ over $aA$, to $g$, which is generic in $G_1a\cap X$ over $aA$. Thus, by Fact \ref{F: vicinities} (3),  it sends $\nu_{G_1\cap Xa^{-1}}(h)$ to $\nu_{G_1a\cap X}(g)$, namely $\left(\nu_{G_1\cap Xa^{-1}}(h)\right)ag^{-1}\equiv  \left(\nu_{G_1a\cap X}(g)\right)g^{-1}$. In addition, since $g$ is generic in $G_1a\cap X$ over $Aa$ and also in $X$ over $A$ it follows from  Lemma \ref{L: restriction of nu to subset} (applied to the injective image of the two sets in $D^n$) that  $\nu_{G_1a\cap X}(g)\equiv G_1a\cap \nu_X(g)$. 
				
				Plugging the above observations into  (\ref{eq:subgroup}) we obtain
				\[\nu_D(G_1)=\left(\nu_{G_1a\cap X}(g)\right) g^{-1}\equiv \left(\nu_X(g)\cap G_1a\right)g^{-1}.\] However, $g$ belongs to the coset $G_1a$, hence $G_1ag^{-1}=G_1$, and we conclude that \[\nu_D(G_1)\equiv \nu_X(g)g^{-1}\cap G_1.\] As $X$ is $D$-critical for $G$ and $g$ is generic in it, it follows that $\nu_X(g)g^{-1}\equiv \nu_D(G) $, which implies the desired result. 
			\end{proof}

			The following general lemma is most likely already known.

			\begin{lemma}\label{L: sub-additivity for groups}
				Let $G$ be a group of finite dp-rank definable in some sufficiently saturated structure $\CM$ and $H_1,H_2\leq G$ be type-definable subgroups. Then \[\tag{*}
				\dpr(H_1)+\dpr(H_2)\le \dpr(H_1\cdot H_2)+\dpr(H_1\cap H_2). 
				\]
			\end{lemma}
			\begin{proof}
				Consider the function $f: H_1\times H_2\to G$ given by $(x_1,x_2)\mapsto x_1x_2$. For $a_i\in H_i$ the function $x\mapsto (a_1x, x^{-1}a_2)$ gives a definable bijection between  $H_1\cap H_2$ and the fiber of $f$ over $a_1a_2$. Thus, setting $d:=\dpr(H_1\cap H_2)$ we see that the fibers of $f$ have constant rank $d$. Since $\dpr(H_1\times H_2)=\dpr(H_1)+\dpr(H_2)$ the inequality (*) follows from sub-additivity. 
			\end{proof}

			We now treat the remaining case.
			
			\begin{lemma}
				Let  $G$ be a definable group in $\CK$,  and $G_1\leq G$ a definable subgroup, and $D=\KOp$. Then $\nu_D(G_1)$ has finite index in $\nu_D(G)\cap G_1$. 
				
				Moreover, there exists a definable subgroup $G_0\le G_1$, $\dpr(G_0)=\adrk(G_1)$,  such that $$\nu_D(G_1)\equiv \nu_D(G)\cap G_0,$$ and in addition the image of $G_0$ inside $G/H_G$ is definably isomorphic to a subgroup of $((K/\CO)^m,+)$, $m=\adrk(G_1)$. 
			\end{lemma}
			\begin{proof} 
				The case where $\adrk(G_1)=0$ is immediate, since $\nu_D(G)\cap G_1$ must be finite. So,  assume that $\adrk(G_1)=m>0$.
				
				By Proposition \ref{P:lasi in K/O p-adic.1}, there exists a definable subgroup $H\leq G$ which is definably isomorphic to $B/C$, where $B\leq (K/\CO)^n$ (for $n=\adrk(G)$) is a ball and $C\leq B$ is finite.
				
				We claim that $\adrk(G_1\cap H)=\adrk(G_1)$: By Lemma \ref{L:locally critical}, for $Y\subseteq G_1$ almost $D$-critical in $G_1$, there is some $g\in G$ with $\adrk(G_1)=\dpr(H\cap gY)=\dpr(g^{-1}H\cap Y)$. The latter set is almost strongly internal to $D$ and contained in a coset of $H\cap G_1$, so $\adrk(H\cap G_1)=\adrk(G_1)$. 
				
				By Proposition \ref{P:general nu of finite index subgroup}, $\nu_D(G_1)\equiv \nu_D(G_1\cap H)$ and $\nu_D(G)\equiv \nu_D(H)$.
				Thus, if we proved that $\nu_D(G_1\cap H)$ had finite index in $\nu_D(H)\cap (G_1\cap H)\equiv\nu_D(H)\cap G_1$ then clearly so does $\nu_D(G_1)$. Also, if we proved that $\nu_D(G_1\cap H)\equiv \nu_D(H)\cap G_0$ for some definable group $G_0\leq H$ then the same would be true if we replaced $\nu_D(G_1\cap H)$ with $\nu_D(G_1)$ and $\nu_D(H)$ with $\nu_D(G)$.

				Therefore,  we may assume that $G=B/C$ and $G_1\leq B/C$ is a definable subgroup. 
				Let $\widehat G_1$ be the preimage of $G_1$ inside $B$.
				
				\begin{claim} It suffices to prove the lemma for $\widehat G_1$ and $B$ (instead of $G_1$ and $G$). 
				\end{claim}
				
				\begin{claimproof}
					Indeed, assume that we know that $\nu_D(\widehat G_1)$ has finite index in $\nu_D(B)\cap \widehat G_1$. Let $\pi:B\to B/C$ be the quotient map. By Proposition \ref{P:nu under finite to one}, $\pi(\nu_D(\widehat G_1))\equiv \nu_D(G_1)$ and $\pi(\nu_D(B))\equiv \nu_D(G)$, so $\nu_D(G_1)$ has finite index in $\pi(\nu_D(B)\cap \widehat G_1)$. However, since $C$ is finite it is a subgroup of $\nu_D(B)$ (Remark \ref{R: how nu looks}(2)), so we must have  $\pi(\nu_D(B)\cap \widehat G_1)\equiv \pi(\nu_D(B))\cap G_1$, and hence $\nu_D(G_1)$ has finite index in $\nu_D(G)\cap G_1$.
					
					As for the moreover clause in the statement of the lemma, assume that there is a definable subgroup $\widehat G_0\leq B$ such that $\nu_D(\widehat G)\equiv \nu_D(B)\cap \widehat G_0$. Taking the image under $\pi$ (using again that $\ker(\pi)\leq \nu_D(B)$), we see that  $\nu_D(\widehat G_1)\equiv \pi(\nu_D(B)\cap \widehat G_0)\equiv \nu_D(G)\cap \pi(\widehat G_0)$, so $G_0:=\pi(\widehat G_0)$ is a definable group with the desired properties.
				\end{claimproof}
				
				Thus, we may assume that $G$ is a ball in $(K/\CO)^n$. So $\nu_D(G)$ is the intersection of all balls around $0$ (Fact \ref{F: concrete description of nu for D-balanced}). If $\dpr(G_1)=\dpr(G)$ then, as $G$ is a subset of $(K/\CO)^n$, by Proposition \ref{P:general nu of finite index subgroup} $\nu_D(G)=\nu_D(G_1)$, and there is nothing to prove. 
				
				So we assume that $\dpr(G_1)<\dpr(G)$. By \cite[Proposition 3.10]{HaHaPeSemisimple}, there exists a natural number $k$ and a coordinate projection $\pi:(K/\CO)^n\to (K/\CO)^m$, for $m=\dpr(G_1)$, such that $\pi\rest p^kG_1$ is injective. By Fact \ref{F: concrete description of nu for D-balanced}, $\nu_D(p^kG_1)\equiv p^kG_1 \cap \nu_D(G)$ and since $p^kG_1$ has finite index in $G_1$ (\cite[Remark 3.11]{HaHaPeSemisimple}) it follows  by Proposition \ref{P:general nu of finite index subgroup} that $\nu_D(p^kG_1)\equiv \nu_D(G_1)$. So we get that $\nu_D(G_1)\equiv p^kG_1\cap \nu_D(G)$ thus the moreover clause of the lemma is proved for $G_0=p^kG_1$. As $p^kG_1$ has finite index in $G_1$, the lemma is proved.
			\end{proof}
			
			Note that the finite index clause in the previous lemma cannot be improved in general due to the existence of finite subgroups in $\KOp$.\\

			Summing up the above results we get: 
			\begin{theorem}\label{T: nu of subgroup}
				Let $G$ be a definable group in $\CK$ and $G_1$ a definable subgroup and $D$ a distinguished sort. Then  $\nu_D(G_1)$ has finite index in $\nu_D(G)\cap G_1$ and: 
				\begin{itemize}
					\item If $D$ is stable or $D=\KOp$ then there exists a definable subgroup $G_0\leq G_1$ almost strongly internal to $D$ with $\dpr(G_0)=\adrk(G_1)$ such that $\nu_D(G_1)\equiv \nu_D(G)\cap G_0$. 
					\item In all other cases, $\nu_D(G_1)\equiv \nu_D(G)\cap G_1$. 
				\end{itemize}
			\end{theorem}
			
			\begin{corollary}\label{C: nu rel def}
				Let $G$ be a definable group in $\CK$ and $G_1$ a definable subgroup. Then $\nu(G_1)$ is relatively definable in $\nu(G)$.
			\end{corollary}
			\begin{proof}
				Let $D_1,D_2,D_3,D_4$ be the different distinguished sorts. By Theorem \ref{T:main}, the definable map $f: (x_1,x_2,x_3,x_4)\mapsto x_1x_2x_3x_4$ is injective on $\prod_{i=1}^4 \nu_{D_i}(G)$. By compactness, there exist definable sets $X_i\sub G$, $\nu_{D_i}(G)\vdash X_i$, such that  $f$ is injective on $\prod_{i=1}^4X_i$. 
				
				By Theorem \ref{T: nu of subgroup}, each $\nu_{D_i}(G_1)$ is relatively definable in $\nu_{D_i}(G)$, so we can find $Y_i\subseteq X_i$ such that $\nu_{D_i}(G)\cap Y_i\equiv \nu_{D_i}(G_1)$.
				
				We claim that $f\left( \prod_{i=1}^4Y_i\right)\cap \nu(G)\equiv \nu(G_1)$.
				
				Suppose $y_i\in Y_i(\widehat\CK)$ are elements such that $f(y_1,y_2,y_3,y_4)=y_1y_2y_3y_4\in \nu(G)(\widehat \CK)$. Since $f(\prod_{i=1}^4\nu_{D_i}(G))\equiv \nu(G)$ it follows, by injectivity, that $y_i\in \nu_{D_i}(G)(\widehat \CK)$ so $y_i \in \nu_D(G_1)(\widehat \CK)$.
				
				For the other direction, if $y\in \nu(G_1)(\widehat \CK)$ then $y=y_1y_2y_3y_4$ for $y_i\in \nu_{D_i}(G_i)(\widehat \CK)=(Y_i\cap \nu_{D_i}(G))(\widehat \CK)$ so $y\in \left(f\left( \prod_{i=1}^4Y_i\right)\cap \nu(G)\right)(\widehat \CK)$.
			\end{proof}

			Although we showed  the statement below for each $\nu_D(G)$ separately, the following remains open for $\nu(G)$.  
			\begin{question}
				Let $G_1\le G$ be definable groups in $\CK$. Does $\nu(G_1)$ have finite index in $\nu(G)\cap G_1$, and  for $D\neq \KOp$, is $\nu(G_1)=\nu(G)\cap G_1$? 
			\end{question}
			
			When $\dpr(G_1)=\dpr(G)$ we prove a stronger statement:

			\begin{proposition}\label{P: nu is local}
				Let $G$ be a group definable in $\CK$, $G_1\le G$ a definable subgroup. Assume that $\dpr(G)=\dpr(G_1)$. Then for any distinguished sort $D$, $\adrk(G)=\adrk(G_1)$ and in particular $\nu_D(G_1)\equiv \nu_D(G)$ and $H_G=H_{G_1}$. As a result, $\nu(G)\equiv \nu(G_1)$.
			\end{proposition}
			
			\begin{proof}   
				
				By Lemma \ref{L: sub-additivity for groups} we get  
				\[
				\dpr(\nu_D(G))+\dpr(G_1)\leq \dpr(G_1\cdot \nu_D(G))+\dpr(\nu_D(G)\cap G_1).
				\] 
				Because $\dpr(G_1)=\dpr(G_1\cdot \nu_D(G))=\dpr(G)$ we get that \[\dpr(\nu_D(G))\leq \dpr(\nu_D(G)\cap G_1)\] so equality must hold and this quantity is equal to $\adrk(G)$.
				
				Since $\nu_D(G)\cap G_1\vdash G_1$ is an almost strongly internal to $D$, of dp-rank $\adrk(G)$, it follows that $\adrk(G_1)=\adrk(G)$ and the rest follows by Proposition \ref{P:general nu of finite index subgroup} and Lemma \ref{L: same H if same adrk}.
			\end{proof}

			\section{Infinitesimals of products}

			Continuing our investigation of the functorial properties of the mapping sending a definable group $G$ to its groups of infinitesimals, we prove that the mapping $G\mapsto \nu_D(G)$ respects direct products. Our first goal is to prove additivity of the almost $D$-ranks. Namely, that for every definable $X_1,X_2$,  $\adrk(X_1\times X_2)=\adrk(X_1)+\adrk(X_2).$ This result was claimed in \cite[Remark 4.19(2)]{HaHaPeGps} for any vicinic sort, but the proof had a gap. Here we provide a proof for $D$, a distinguished sort in our settings.

			For $Y\sub X_1\times X_2$, and $t\in X_1$, as usual, we write $Y_t=\{y\in X_2:(t,y)\in Y\}$. We make use of the fact that $Y_t$ is in definable bijection with the set $\{t\}\times Y_t\sub Y$.
			\begin{lemma}\label{L:finite cover}
				Let $\CM$ be a first order structure and $D$ a definable set of finite algebraic dp-rank and $T$ any definable set.
				Assume that $X\sub T\times D^n$ is $A$-definable,  such that $\dpr(X_t)\leq d$ for all $t\in T$. Then there exists an $A$-definable partition $X=\coprod_{i=1}^\ell X_i $
				and   for each $i$ there are $m_i\in \mathbb N$ and a coordinate projection $\pi_i:D^n\to D^d$,
				such that for each $t\in T$, $\pi\rest X_{i,t}$ is at most $m_i$-to-one. 
			\end{lemma}
			\begin{proof}
				This is a standard compactness argument.  Let $\CC$ be the collection of all formulas $\varphi(t,x)$ over $A$ with  $(t,x)\in T\times D^n$, such that there is $m\in \mathbb N$ and a coordinate projection $\pi:D^n\to D^d$, so that  for every $b\in T$ the restriction of $\pi$ to $\varphi(b,D^n)$ is $m$-to-1.  We consider the partial type 
				\[\Sigma(t,x):=  \{(t,x)\in X\}\cup \{\neg \varphi(t,x): \varphi\in \CC \},\] and claim  it is inconsistent. 
				
				Indeed, assume towards a contradiction that $(t_0,a)$ realizes  $\Sigma$. As $a\in X_{t_0}$, $\dpr(a/At_0)\leq d$, and since dp-rank is algebraic, there is $a_1\subseteq a$ of length $d$
				such that $a\subseteq  \acl(At_0a_1)$. Let $\pi_1: D^n\to D^d$ be the projection onto the $a_1$-coordinates of $a$ and let $a_0$ be the image of $a$ under the projection $\pi_0$ onto the remaining coordinates. Let $\psi(t_0,x_0,x_1)\in \tp(a_0,a_1/At_0)$ be such that $\models (\exists^{= m}x_0) \psi(t_0,x_0, a_1)$ for some $m\in \Nn$.  
				Let $$\phi(t,x):= \psi(t,\pi_0(x),\pi_1(x))\, \wedge\,  \exists^{= m} y \, \psi(t,\pi_1(x),y).$$ 
				Thus, $\phi \in \CC$ and $\models \phi(t_0,a)$, contradiction.

				So $\Sigma$ is inconsistent, and  by compactness
				we obtain finitely many definable sets $Z_1,\ldots, Z_{\ell}\sub X$, whose union covers $X$, and for each $i$, a coordinate projection $\pi_i:D^n\to D^d$ whose restriction to each $Z_{i,t}$ is  finite-to-one.
				Finally, to obtain a partition of $X$ we replace $Z_2$ by $Z_2\setminus Z_1$, $Z_3$ by $Z_3\setminus (Z_1\cup Z_2)$ etc. 
			\end{proof}

			\begin{lemma}\label{L:equivalence relation dp-rank}
				Let $\CM$ and $D$ be as in the previous lemma, and $X\sub T\times D^n$. Assume that $\dpr(\bigcup_{t\in T}X_t)=k$ and for each $t\in T$, $\dpr(X_t)\leq d$.

                Then there exists a definable set $C\sub X$ with $\dpr(\bigcup\limits_{t\in T} C_t )\ge k-d$ such that for every $t\in T$,  $C_t$ is finite (possibly empty). 
				
			\end{lemma}
			\begin{proof}
				Assume everything is definable over some $A$. 
				By Lemma \ref{L:finite cover}, there is a definable partition $X:=\coprod_{i=1}^\ell Z_i$, and for each $i=1,\ldots, \ell$, a  coordinate projection $\pi_i: D^n \to D^d$ such that for every $t\in T$, the fibers of $\pi_i\rest Z_{i.t}$  are finite. 

We replace $X$ with one of the $Z_i$ for which $\dpr(\bigcup_{t\in T}Z_{i,t})=k$. Thus, we may assume that for each $t\in T$, $\pi:X_t\to D^d$ is finite-to-one.

 Let $a\in \bigcup_t X_t$ be $A$-generic and let $a_0=\pi(a)$. Sub-additivity implies that $\dpr(a/Aa_0)\geq k-d$. Let $C=\{(t,b)\in X:\pi(b)=a_0\}$. By the choice of initial partition of $X$, for each $t\in T$, $C_t$ is finite. 
  Since $\bigcup_{t\in T} C_t$,  is defined over $Aa_0$ and contains $a$ it follows that $\dpr(C)\geq k-d$.\qedhere


				
			\end{proof}
			
			\begin{remark} In the above lemma consider the special case when the $X_t$ are pairwise disjoint. In this situation, the projection of $X$ in $T$ can be viewed as the quotient of $\bigcup_t X_t$ by an equivalence relation whose classes are the $X_t$. The resulting set $C$ is aimed to be an "almost section'' in the sense that it chooses, on a subset of $X$ of full dp-rank  a finite set of representatives of each class it intersects. However, if all the $X_t$ have the same dp-rank as $X$, then the set $C$ itself could be finite -- very far from an actual section. 
			\end{remark}
			
			\begin{proposition}\label{P:subadditivity of drk and adrk}
				Let $X_1$ and $X_2$ be definable sets in $\CK$, and $Y\subseteq X_1\times X_2$ be definable. Then,
				\begin{enumerate}
					\item (on $\drk(Y)$) 
					
					(i) For all distinguished sorts $D$, \[\drk(Y)\leq \adrk(\pi_1(Y))+\max\{\drk(Y_a): a\in X_1\}.\]
					
					(ii) If $D\neq \KOp$ then \[\drk(Y)\leq \drk(\pi_1(Y))+\max\{\drk(Y_a): a\in X_1\}.\]

					\item (on $\adrk(Y)$) For all distinguished $D$, \[\adrk(Y)\leq \adrk(\pi_1(Y))+\max\{\adrk(Y_a): a\in X_1\}.\]
				\end{enumerate}
			\end{proposition}
			\begin{proof}
				Let $Y\subseteq X_1\times X_2$ be a definable subset. Note that it is sufficient to prove the above inequalities for 
				$Y_1\sub Y$ which is (almost) strongly internal to $D$, whose dp-rank realizes the (almost) $D$-rank of $Y$. 
				Thus we may assume that $Y$ is (almost) strongly internal to $D$. We treat the two cases uniformly, so we fix $f:Y\to D^n$,  an $r$-to-$1$ definable function and when $Y$ is strongly internal to $D$, we consider the case $r=1$.
				
				Let $T=\pi_1(Y)$, $k:=\dpr(Y)$ and $d=\max\{\dpr(Y_t):t\in T\}$.

				We shall apply Lemma \ref{L:equivalence relation dp-rank}, to 
				the set 
				$$W=\{(t,f(t,x))\in T\times D^n:(t,x)\in Y\}.$$
				We have $f(Y)=\bigcup_t W_t$, and since $f$ is finite-to-one, $\dpr(f(Y))=k$.
				
				Since $f$ is at most $r$-to-one, each $y\in f(Y)$ belongs to at most $r$-many $W_t$ (as $t$ varies).


                 Let $C\sub W$ be as provided by Lemma \ref{L:equivalence relation dp-rank}. Namely, 
				 $\dpr(C)\geq k-d$ and  each fiber $C_t$ is finite. Thus, the projection of $C$ on the $T$-coordinate is finite-to-one. Let $S$ be the projection of $C$ into the $D^n$-coordinate. Since each $x\in S$ belongs to at most $r$-many $C_t$, the projection of $C$ onto  $S$ is a surjection, with fibers of size at most size $r$.
				

				\textbf{\underline{Case 1:}} Assume that $r=1$, namely $f$ witnesses $Y$ being strongly internal to $D$.  It follows from the above that  $C$ projects bijectively onto  $S$ and hence it is the graph of  a finite-to-one function from $S\sub D^n$ onto some $T_1\sub T$, with $\dpr(T_1)=\dpr(S)\geq k-d$. By \cite[Lemma 2.9(1)]{HaHaPeGps}, if  $D$ is an SW-uniformity and by \cite[Lemma 4.3]{HaHaPeGps} if $D$ is $\bk$ or $\Gamma$, there exists a definable subset $T_2\subseteq T_1$ with $\dpr(T_2)=\dpr(T_1)$ which is strongly internal to $D$. Thus, for $D\neq \KOp$, $\drk(T)\geq \dpr(T_2)\geq k-d$, and since each $Y_a$ is strongly internal to $D$ then for all $a\in \pi_1(Y)$, $\dpr(Y_a)=\drk(Y_a)$. Combining everything we get, for $D\neq \KOp$,
				\[\drk(Y)=k= (k-d)+d\leq \drk(T)+ d= \drk(\pi_1(Y))+\max\{\drk(Y_a):a\in \pi_1(Y)\}.\]

				When  $D=\KOp$ we can find, by \cite[Lemma 3.9(2)]{HaHaPeGps}, a subset $T'\sub T$, $\dpr(T')\geq k-d$, with $T'$ \textit{almost strongly internal} to $D$.  Thus $\adrk(T)\geq \dpr(T')= k-d$. Putting together with the above, we get for all $D$,  
				\[\drk(Y)\leq \adrk(T)+d=\adrk(\pi_1(Y))+d.\]
				
				This proves (1)(i),(ii).
				
				\textbf{\underline{Case 2:}} Assume that $r> 1$. The projection of $C$ onto $S$ witnesses that $C$ is almost strongly internal to $D$. The projection of $C$ into the $T$-coordinate is a finite-to-one map onto a set $T_1\sub T$, hence $\dpr(T_1)=\dpr(C)\geq k-d$. By \cite[Lemma 2.9(2), Lemma 3.9(2), Lemma 4.3]{HaHaPeGps}, there exists a definable subset $T_2\subseteq T_1$ with $\dpr(T_2)=\dpr(T_1)$ which is almost strongly internal to $D$. Thus $$\adrk(T)\geq \adrk(T_1)\geq \dpr(T_2)\geq k-d.$$ Therefore, for all $D$,
				\[\adrk(Y)= k-d+d\leq \adrk(T)+d\leq \adrk(\pi_1(Y))+\max\{\adrk(Y_a):a\in \pi_1(Y)\},\]
				as required. This proves (2).
			\end{proof}
			
			The following is an immediate consequence of Proposition \ref{P:subadditivity of drk and adrk}, taking $Y=X_1\times X_2$.
			\begin{corollary}\label{C:adrk and drk of products}
				Let $X_1$ and $X_2$ be two infinite definable sets in $\CK$. Then, 
				
				\begin{enumerate}
					\item For all distinguished sorts D,
					
					(i) if  $\drk(X_1)=\adrk(X_1)$  then $\drk(X_1\times X_2)=\drk(X_1)+\drk(X_2).$

					(ii)  If $D\neq \KOp$ then $\drk(X_1\times X_2)=\drk(X_1)+\drk(X_2).$
					
					\item For all distinguished sorts $D$, $\adrk(X_1\times X_2)=\adrk(X_1)+\adrk(X_2).$

				\end{enumerate}
			\end{corollary}
			
			The following is a direct corollary of Corollary \ref{C:adrk and drk of products}
			\begin{corollary}\label{C:product of balanced}
				Let $G_1$ and $G_2$ be definable groups in $\CK$, locally almost strongly internal to $D$.
				\begin{enumerate}    
					\item $\adrk(G_1 \times G_2)=\adrk(G_1)+\adrk(G_2)$ and if either $G_1$ or $G_2$ is $D$-balanced then 
					$\drk(G_1\times G_2)=\drk(G_1)+\drk(G_2).$
					\item If both $G_1$ and $G_2$ are $D$-balanced then so is $G_1\times G_2$.
				\end{enumerate}
			\end{corollary}

			\begin{proposition}\label{P: H of products}
				Let $\CK$ be a $p$-adically closed field and $D=K/\CO$. Let $G_1$ and $G_2$ be definable groups in $\CK$, both locally almost strongly internal to $D$. Then $H_{G_1\times G_2}=H_{G_1}\times H_{G_2}$.
			\end{proposition}
			\begin{proof}
				For $i=1,2$ let $d_i=\adrk(G_i)$ and let $H_i=H_{G_i}$. Then by Corollary \ref{C:adrk and drk of products}, $\adrk(G_1\times G_2)=d_1+d_2$. Let $G=G_1\times G_2$ and identify each $G_i$ with a subgroup of $G$. Since $G/(H_1\times H_2)$ is $D$-balanced (by Corollary \ref{C:product of balanced}), Proposition \ref{Canonical H} implies that $H_G\leq H_1\times H_2$. It is sufficient to show that $H_i \sub H_G\cap G_i$, for $i=1,2$; we treat the case $i=2$.
				
				Let $Y\subseteq G/H_G$ be a $D$-critical subset, so its preimage $\widehat Y\subseteq G=G_1\times G_2$, is almost strongly internal to $D$ and $\dpr(\widehat Y)=d_1+d_2$. Together with Proposition \ref{P:subadditivity of drk and adrk} (2), we conclude that 
				\[d_2\leq \max\{\adrk(\widehat Y_a):a\in G_1\}.\]
				Let $a\in G_1$ be an element with $\adrk(\widehat Y_a)\geq d_2$. Since $\widehat Y_a$ is almost strongly internal to $D$, $\dpr(\widehat Y_a)=\adrk(\widehat Y_a)\geq d_2$. But $\widehat Y_a$ is a subset of $G_2$, hence $\dpr(\widehat Y_a)=d_2=\adrk(G_2)$. 
				
				Consider the group $G_2/(H_G\cap G_2)$. By Fact \ref{F:finite index same adrk}, $\adrk(G_2/(H_G\cap G_2))=\adrk(G_2)=d_2$. The image of $\widehat Y_a$ in $G/(H_G\cap G_2)$ is in definable bijection with its image inside of $G/H_G$ but the image in the latter lies inside $Y$, which is strongly internal to $D$.  Hence, the image of $\widehat Y_a$ inside $G/(H_G\cap G_2)$ is strongly internal to $D$. It follows that $d_2=\drk(G_2/(H_G\cap G_2)=\adrk(G_2/(H_G\cap G_2))$, so $G_2/(H_G\cap G_2)$ is $D$-balanced. By Proposition \ref{Canonical H}, $H_2\subseteq H_G\cap G_2$
			\end{proof}
			
			\begin{proposition}\label{P:nu in products}
				Let $G_1$ and $G_2$ be two definable groups in $\CK$ locally almost strongly internal to a distinguished sort $D$. Then $\nu_D(G_1\times G_2)\equiv \nu_D(G_1)\times \nu_D(G_2)$. Consequently, also $\nu(G_1\times G_2)\equiv \nu(G_1)\times \nu(G_2)$. 
			\end{proposition}
			\begin{proof} Let $G=G_1\times G_2$.
				We use the characterization of $\nu_D(G)$ (and similarly each $\nu_D(G_i)$), from  Proposition \ref{P:characterizatio of nu_D in general}, as the collection of all $XX^{-1}$ for $X\sub G$ almost strongly internal to $D$ with $\dpr(X)=\adrk(G)$.
				
				Assume, for $i=1,2$, that $X_i\sub G_i$ is almost strongly internal to $D$ and $\dpr(X_i)=\adrk(G_i)$. For  $X=X_1\times X_2$, $X$ is almost strongly internal to $D$ and, by Corollary \ref{C:adrk and drk of products}, $\dpr(X)=\adrk(G)$. We have 
				$XX^{-1}=(X_1X_1^{-1})\times (X_2X_2^{-1})$  and therefore $\nu_D(G)\vdash \nu_{D}(G_1)\times \nu_D(G_2)$.
				
				Conversely, assume that $X\sub G$ is almost strongly internal to $D$ with $\dpr(X)=\adrk(G)$. Since $\nu(G)$ is a type-definable group there is some definable subset $Y\subseteq G$, almost strongly internal to $D$ with $\dpr(Y)=\adrk(G)$, such that $(YY^{-1})(YY^{-1})\subseteq XX^{-1}$. By Proposition \ref{P:subadditivity of drk and adrk}, $Y$ contains a fiber $\{a\}\times Y_2$, for $Y_2\sub G_2$ almost strongly internal to $D$, with $\dpr(Y_2)=\adrk(G_2)$, and by symmetry, it also contains a set $Y_1\times \{b\}$ with $Y_1\sub G_1$ almost strongly internal to $D$, $\dpr(Y_1)=\adrk(G_1).$ It follows that $ (Y_1Y_1^{-1})\times (Y_2Y_2^{-1})\sub (YY^{-1})(YY^{-1})\subseteq XX^{-1}$, and hence $\nu_D(G_1)\times \nu_D(G_2)\vdash \nu_D(G)$. By Theorem \ref{T:main} this implies $\nu(G_1)\times \nu(G_2)\vdash \nu(G)$. 
			\end{proof}

			\bibliographystyle{plain}
			\bibliography{harvard}

		\end{document}